\documentclass[11pt]{article}
\usepackage{amsmath, amssymb, amsfonts, amsthm, latexsym, hyperref, graphicx}
\usepackage{pst-plot}
\usepackage{pstricks,comment}
\usepackage{enumerate}
\usepackage{bbm}
\usepackage[english]{babel}
\usepackage[T1]{fontenc}
\usepackage{lmodern}	% font definition
\usepackage{amsmath}	% math fonts
\usepackage{amsthm}
\usepackage{amsfonts}

\usepackage{tikz}
\usetikzlibrary{decorations.pathmorphing} % noisy shapes
\usetikzlibrary{fit}					% fitting shapes to coordinates
\usetikzlibrary{backgrounds}	% drawing the background after the foreground
\usetikzlibrary{arrows.meta}
\usetikzlibrary{arrows}
\usetikzlibrary{decorations.markings}
\usetikzlibrary{positioning}
\usetikzlibrary{calc}
\usetikzlibrary{patterns,angles,quotes}

\newtheorem*{uthm}{Theorem}
\newtheorem{thm}{Theorem}

\newtheorem{lem}{Lemma}[section]
\newtheorem{obs}[lem]{Observation}

\newtheorem{propos}[lem]{Proposition}

\newtheorem{defin}[lem]{Definition}

\newtheorem{conjecture}{Conjecture}

\newcommand{\degr}{\text{deg}}

\newcommand{\QR}{\text{QR}}

\newcommand{\R}{\mathbb{R}}
\newcommand{\C}{\mathbb{C}}
\newcommand{\N}{\mathbb{N}}
\newcommand{\cP}{\mathcal P}
\newcommand{\M}{\mathcal{M}}
\newcommand{\T}{\mathbb{T}}
\newcommand{\ty}{\text{Ty}}

\newcommand{\Z}{\mathbb{Z}}

\newcommand{\ind}{\mathbbm{1}}

\DeclareMathOperator{\dn}{dn}

\title{Drawing outerplanar graphs using thirteen edge lengths}
\author{Ziv Bakhajian\thanks{Israel Arts and Science Academy, Chaim E. Kolitz rd., Jerusalem, Israel. Email: zivbakhajian@gmail.com}
\and Ohad N. Feldheim \thanks{Einstein Institute of Mathematics,
The Hebrew University of Jerusalem,
Givat Ram, Jerusalem, Israel. Email: ohad.feldheim@mail.huji.ac.il}}

\begin{document}
\maketitle

\begin{abstract}
We show that every outerplanar graph $G$ can be linearly embedded in the plane
such that the number of distinct distances between pairs of adjacent vertices is at most thirteen and there is no intersection between
the image of a vertex and that of an edge not containing it.

This extends the work of Alon and the second author, where only overlap between vertices was disallowed, thus settling a problem posed by Carmi, Dujmovi\'{c}, Morin and Wood.
\end{abstract}

\section{Introduction}

The subject of this paper are \emph{linear embedding}s of a graph $G=(V,E)$. Such an embedding is a map $\psi:V\to \C$, where we view the open interval $(\psi(u),\psi(v))$ as the image of the edge $(u,v)\in E$ through $\psi$.
The length of this interval is called the \emph{edge length} of $(u,v)$ in the embedding $\psi$.

A \emph{degenerate drawing} of a graph $G$ is a linear embedding in which the
images of all vertices are distinct. A \emph{drawing} of $G$ is a degenerate
drawing in which the image of every edge is disjoint from the image of every
vertex. Observe that in both cases edges are allowed to cross each other so that the embedding is not planar (see Figure~\ref{fig: T star}).
The  \emph{(degenerate) distance number} of a graph is the minimum number of distinct edge lengths in a (degenerate) drawing of $G$.

An \emph{outerplanar} graph is a graph that can be embedded in the plane
without edge crossings so that all vertices lie on the boundary of
the unbounded face of the embedding. In \cite{Pvpd}, Carmi, Dujmovi\'{c},
Morin and Wood ask whether the distance number of outerplanar graphs is
uniformly bounded. We answer this question to the affirmative, showing that the distance number of outerplanar graphs is at most $13$. %, and that the distance number of such graphs is at most $9$. Both results

\begin{thm}\label{thm: 13 degenerate}
Every outerplanar graph has distance number at most $13$.
\end{thm}

This we obtain by providing an explicit drawing of the universal infinite outerplanar graph (See Definition~\ref{def:inftree} below),
for generic edges-lengths, subject to certain triangle inequalities.

\section{Background and Motivation}

The notions of distance number and degenerate distance number of a graph
were introduced by Carmi, Dujmović, Morin and Wood in
\cite{Pvpd} in order to generalize several well studied problems. 
These include the problem of estimating the minimum possible number of distinct distances between $n$ points in the plane (see \cite{Erd}, \cite{GK}) and
the minimum number of distances between $n$ non-co-linear points in the plane (\cite[Theorem 13.7]{Smz}).

The notion of degenerate distance number also generalizes the notion of a unit distance graph (equivalent to degenerate distance number 1), while distance number puts an additional constraint that vertices and edges do not overlap. For additional discussion concerning these problems and their relation to distance numbers, see \cite{AF,Pvpd}.

After introducing the notions of distance number and degenerate distance number, Carmi, Dujmović, Morin and Wood
studied in \cite{Pvpd} the behavior of bounded degree graphs with respect to these notions. They
show that the degenerate distance number of certain graphs with maximum degree five can be arbitrarily large, giving a polynomial lower-bound
for other graphs whose degree is at most seven. They also provide a $c
\log(n)$ upper-bound on the distance number of bounded degree graphs with
bounded treewidth (and in particular of outerplanar graphs). There, the authors raised the question as to whether the (degenerate) distance number of outerplanar graphs is uniformly bounded. Alon and the second author~\cite{AF} gave a partial answer by proving that outerplanar graphs have degenerate distance number at most $3$, however they observe that their construction does not result in a drawing.  Here we extend their result showing that outerplanar graphs have distance number at most $13$. The question of finding the exact maximum distance number of outerplanar graphs remains open.

\section{Proof outline}

Our proof of Theorem~\ref{thm: 13 degenerate} is obtained in four steps, whose rough outline we provide here.

\textbf{Reduction to drawing the rhombus tree.} Firstly, we recall that outerplanar graphs are all sub-graphs of a
universal cover $T^*$ called the (binary) \emph{triangle tree}, generated by starting from a single base triangle with a marked base edge and iteratively gluing additional triangles along any unglued side of a triangle except for the marked base edge, ad infinitum.
We observe that this graph could also be represented as a trinary tree $H^*$ whose nodes are copies of a rhombus $H$, where two nodes are connected if they are glued to each other (see Figure~\ref{fig: T star}).
Denoting $\pi:H^*\times H\to T^*$ for the natural mapping of $H^*$ to $T^*$, Theorem~\ref{thm: 13 degenerate} reduces to showing that the images of every edge of $H$ through $\pi((H^*,H))$ is of one of $13$ distances.

\textbf{Encoding the rhombi.} To describe each rhombus of $H^*$, we encode it using the sequence of descending steps taken to reach it from the root rhombus. These correspond to \emph{left (right) steps}, where we move to a rhombus glued along the edge to the left (right) of the base edge and \emph{forward steps}, where we move to a rhombus glued along the edge parallel to the base edge. These steps are depicted in Figure~\ref{fig: T star}.
As every vertex in $T^*$  appears in infinitely many rhombi (i.e., images of nodes of $H^*$), we use a canonical one-to-one mapping between vertices and a subset of the nodes, introduced in \cite{AF} (see Observation~\ref{obs: properties of rho-q}).

\textbf{Polynomial embedding.} We define a polynomial mapping $\psi$ which maps every vertex of $T^*=\pi(H^*)$ to a polynomial with integer coefficients in $\C[(x_i)_{i\in I}]$ where $|I|=12$. For every choice of parameters $(x_i)_{i\in I}$ in the unit circle, our mapping $\psi$ reduces to a mapping $\psi_x$ from $T^*$ to the complex plane, such that
every rhombus in $H^*$ is mapped through $\psi_x\circ\pi$ to a rhombus in $\C$ with side length $1$ and diagonal $|x_i-1|$. We call $i$ the \emph{type} of the rhombus.
It is our purpose to show that for almost every choice of $(x_i)_{i\in I}$ in the unit circle, the mapping  $\psi_x$ is a drawing, i.e., the images
of vertices of $T^*$  through $\psi_x$ do not coincide nor do they intersect with the image of any edge.

\textbf{Avoiding intersections.}  %Demonstrating that $\psi$ results in a drawing for almost every ${x_i}_{i\in T}$ involves some technical combinatorial arguments which may hide the ideas behind them.
%To ease the comprehension of these ideas, we outline here a simpler result, showing that forbidden intersections cannot occur between vertices of rhombi that could be reached without taking any right turns on the tree.
In order to show that image of distinct vertices do not coincide for almost every $x$, it suffices to show that the corresponding polynomials do not coincide. This is due to the fact that any two polynomials in $x$ that do not coincide take distinct values for almost every $x$ on the unit circle.
Showing that a vertex and an edge cannot coincide is more challenging, and is the main innovation of the paper. We show that if a vertex $v$ lies on the line extending the image of an edge $(u,w)$ for more than a zero measure set of assignments -- then $\frac{\psi_x(v)-\psi_x(w)}{\psi_x(u)-\psi_x(w)}$ is
a symmetric Laurent polynomial (i.e., a polynomial divided by a monomial with the same coefficients for $M$ and $M^{-1}$). Our choice for the type of each rhombus is made to guarantee that this cannot occur,
unless the vertex is obtained by taking forward steps from a rhombus containing the edge or vice versa (in which case, the vertex and the edge cannot coincide).

\section{Preliminaries}\label{sec:4}
%
%{\bf Marked graphs, and their gluing.} A marked graph is what we call a tuple
%$(G,e)$, where $G$ is an oriented graph and $e\in E_G$ . Let
%$(G_1,e_1),(G_2,e_2)$ be a pair of marked graphs. We define $G_1\glue{e}G_2$
%for $e\in G_1$ to be the marked graph $(G,e_1)$ where $G$ is the union of
%$G_1$ and $G_2$, identifying $e$ and $e_2$.

In this section we repeat many of the notations, definitions and observations used in \cite{AF}.
These are included in order to make the paper self-contained. Throughout we omit separating commas from sequences descriptions and denote concatenation of sequences by product.

{\bf Outerplanarity, $\Delta$-trees and $T^*$.} 
We shall use the well-known fact that a Graph is outerplanar if and only if it has a planar linear embedding mapping its vertices to any set of distinct
roots of unity (see e.g., \cite[Theorem 4.10.3]{BA}) . Such a graph is said to be a \emph{complete outerplanar graph} if all of the bounded faces of such an embedding
are triangular. A \emph{triangulation} of an outerplanar graph is simply a complete outerplanar graph containing it. It is a well known fact that every outerplanar graph can be triangulated.

Let $\Delta$ be the triangle graph, that is, a graph on three vertices $v_0$,
$v_1$, and $v_2$, whose edges are $(v_0,v_1)$, $(v_0,v_2)$ and $(v_1,v_2)$. A graph is
said to be a \emph{$\Delta$-tree} if it can be generated from $\Delta$ by
iterations of adding a new vertex and connecting it to both ends of some
\emph{external edge} $(v,u)$ other than $(v_0,v_1)$, that is, a edge which no triangle has been glued to previously. We call this procedure the \emph{gluing} of a triangle to $(v,u)$, saying that the edge connected to $v$ is its \emph{left} edge and the one connected to $u$ is its \emph{right} edge.
It is a classical fact, which can be proved using induction, that any complete
outerplanar graph is a $\Delta$-tree. 

The adjacency graph of the bounded faces of such a graph forms a binary tree, that is -- a rooted tree of maximal degree 3.
By repeating the above gluing procedure ad infinitum, leaving no external edges, one obtains $T^*$, the infinite complete outerplanar graph, which contains all  $\Delta$-trees as subgraphs. Note that this graph is not locally finite. 
Formally, writing $[m]=\{1,\dots,m\}$,
\begin{defin}\label{def:inftree}
The \textbf{infinite complete outerplanar graph} is the graph  $T^*=(V,E)$ where \linebreak
$V=\{(a_i)_{i\in[n]}\ :\ a_i\in \{0,1\}, n\in \N, a_1=0, a_2=1,a_3=0\}$
and $(a,b)\in E$ for  $a=(a_i)_{i\in[m]}$ and $b=(b_i)_{i\in[n]}$ with  $m\le n$ whenever either
\begin{itemize}
\item $b_{m+1}=0$ and, for all $i>m+1$ we have $b_{i}=1$;
\item or $b_{m+1}=1$ and for all $i>m+1$ we have $b_{i}=0$ .
\end{itemize}
The \textbf{root edge} of the graph is the edge $(0,01)$.
\end{defin}
The root triangle of this tree thus consist of the vertices $0,01,010$. The sequence $p$ describing a particular vertex $v\notin\{0,01,010\}$, consists, after the initial $010$, of sequential instructions whether we should glue a new vertex on the left edge of the previously glued triangle (indicated, say, by 0) or on the right edge (indicated by 1). The vertex added in the final step corresponds to $v$.
Since every $\Delta$-tree is a finite subgraph of $T^*$, we deduce that every outerplanar graphs is also a finite subgraphs of $T^*$. Thus, Theorem~\ref{thm: 13 degenerate} reduces to the following proposition.
\begin{propos}\label{prop: 13 degenerate}
For almost every set $\{a_i\in (0,1)\}_{i \in [13]}$, the graph $T^*$ has a
drawing using only edge lengths from the set $\{a_i\}_{i\in [13]}$.
\end{propos}

{\bf The rhombus graph $H$, Covering $T^*$ by rhombi.} In order to prove the
above proposition, we construct an explicit embedding of $T^*$ in $\C$. To do
so we introduce a covering of $T^*$ by copies of a particular directed graph
$H$ which we call a \emph{rhombus}. We then embed $T^*$ into $\C$, one copy
of $H$ at a time.

  \begin{figure}[htb!]
   \centering%
\begin{tikzpicture}
\begin{scope}[ultra thick, gray,  decoration={markings, mark=at position 0.65 with {\arrow[scale=3/2]{stealth}}}]
\coordinate[label={[xshift=-5mm,yshift=-3mm]\large $v_0$}] (a) at (0,0);
\coordinate[label={[xshift=-5mm,yshift=-3mm]\large $v_2$}] (b) at (2*0.602,2*0.798); %cos,sin
\coordinate[label={[xshift=5mm,yshift=-3mm]\large $v_3$}] (c) at (2+2*0.602,2*0.798);
\coordinate[label={[xshift=5mm,yshift=-3mm]\large $v_1$}] (d) at (2,0);
\draw[postaction={decorate}] (a) -- (b);
\draw[postaction={decorate}] (b) -- (c);
\draw[postaction={decorate}] (d) -- (b);
\draw[postaction={decorate}] (d) -- (c);
\draw[postaction={decorate}] (a) -- (d);
\filldraw[black] (a) circle (2pt);
\filldraw[black] (b) circle (2pt);
\filldraw[black] (c) circle (2pt);
\filldraw[black] (d) circle (2pt);
\end{scope}
\end{tikzpicture}
   \caption{The rhombus graph $H$.}
   \label{fig: H}
   \end{figure}
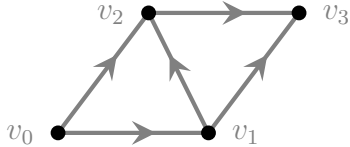

The \emph{rhombus} directed graph $H$, is defined to be the graph satisfying
$V_H=\{v_0,v_1,v_2,v_3\}$ and $E_H=\{(v_0,v_1),(v_0,v_2),(v_2,v_3),(v_1,v_3),(v_1,v_2)\}$.
We call $v_0$ and $(v_0,v_1)$ the \emph{base vertex} and \emph{base edge} of $H$, respectively.

Define the \emph{rhombus tree} $H^*$ to be the infinite directed rooted trinary tree whose \emph{nodes}
are copies of $H$. Labeling the three directed \emph{arcs} emanating from every node by
$0$ (corresponding to $(v_0,v_2)$), $1$ (corresponding to $(v_2,v_3)$) and $2$  (corresponding to $(v_1,v_3)$), 
we denote the root node of $H^*$ by $H^*_{1}$ and every other node by $H^*_a$ where $a$ is a trinary sequence starting with $1$, 
describing the path to it from the root.
%We write $L(a)$ for the label of an arc $(N,M)\in E(H^*)$. 
Let $N$ be a node of $H^*$, and let $(v_i,v_j)\in E_H$; we call a pair $(N,v_i)$
a \emph{vertex} of $H^*$, and a pair $(N,(v_i,v_j))$, an \emph{edge} of $H^*$. Notice the
distinction between arcs of $H^*$ and edges of $H^*$, and the distinction
between nodes and vertices. We denote the collection of all vertices of $H^*$ by $V(V(H^*))$.

There exists a natural map $\pi: V(V(H^*)) \to V(T^*)$ which maps the vertices of each node of $H^*$ to the vertices of a pair of adjacent triangles of
$T^*$ and which is faithful in the following sense. If $N$ is connected to $M$ through an arc labeled by a trinary digit corresponding to an edge $v_iv_j$, then
$\pi((N,v_i))=\pi((M,v_0))$ and $\pi((N,v_j))=\pi((M,v_1))$. The $v_0v_1$ edge of $H^*_1$, the root node of $H^*$, is mapped to the root edge of $T^*$.
To simplify our notation, from now and on we abridge $\pi((N,v))$ to $\pi(N,v)$.

We also provide a formal definition of the mapping $\pi$, as follows.
\begin{defin}\label{def:inftree}
Let $(a_i)_{i\in [n]}$ with $a_i\in{\{0,1,2\}}$ and $a_1=1$. 
$\pi:V(V(H^*))\to V(T^*)$ is then defined by writing
\[(u_0,u_1)=\begin{cases} 
\big((H^*_{(a_i)_{i\in [n-1]}},v_0),(H^*_{(a_i)_{i\in [n-1]}},v_2)\big) & a_n=0\\
\big((H^*_{(a_i)_{i\in [n-1]}},v_2),(H^*_{(a_i)_{i\in [n-1]}},v_3)\big) & a_n=1\\
\big((H^*_{(a_i)_{i\in [n-1]}},v_1),(H^*_{(a_i)_{i\in [n-1]}},v_3)\big) & a_n=2
\end{cases}\]
and recursively setting,
\begin{align*}
\pi(H^*_{a},v_0)&=\pi(u_0)\\
\pi(H^*_{a},v_1)&=\pi(u_1)\\
\pi(H^*_{a},v_2)&=\begin{cases}
\pi(u_1)0 &  a_n\in \{0,1\}\\
\pi(u_1)1 & a_n =2
\end{cases}\\
\pi(H^*_{a},v_3)&=\pi(H^*_{a},v_2)1,
\end{align*}
with 
$\pi(H^*_{1},v_0)=0$ and $\pi(H^*_{1},v_1)=01.$
\end{defin}

We remark, that by writing $b(a)=(b_i)$ for a sequence produced from $a=(a_i)$ by replacing every $1$ along the sequence
with $10$ and every $2$ with $11$, we can provide a direct simple translation formula 
\begin{align}
\pi(H^*_{a},v_2)&=0b(a)\notag \\
\pi(H^*_{a},v_3)&=0b(a)1.\notag
\end{align}

In the rest of the paper we extend $\pi$ naturally to edges and subgraphs. A portion of $T^*$ and its covering by
$H^*$ through $\pi$ are depicted in Figure~\ref{fig: T star}.

  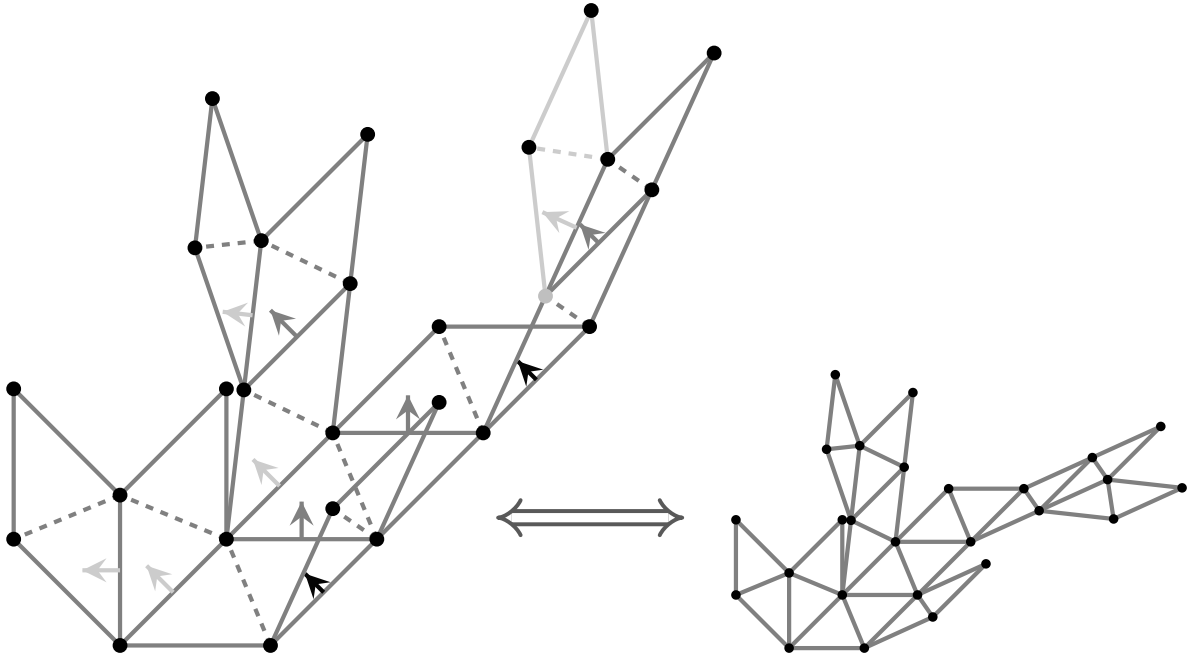
\begin{figure}[htb!]

   \centering%
   \begin{tikzpicture}
\begin{scope}[ultra thick, gray,  decoration={markings, mark=at position 1 with {\arrow[scale=3/2]{stealth}}}]
\coordinate (a) at (0,0);
\coordinate (b) at (2*0.707,2*0.707); %cos,sin
\coordinate (c) at (2+2*0.707,2*0.707);
\coordinate (d) at (2,0);
\coordinate (e) at (4*0.707,4*0.707);
\coordinate (f) at (2+4*0.707,4*0.707);
\coordinate (g) at (2*0.707+2*0.116,2*0.707+2*0.993);
\coordinate (h) at (4*0.707+2*0.116,4*0.707+2*0.993);
\coordinate (i) at (2*0.707+4*0.116,2*0.707+4*0.993);
\coordinate (j) at (4*0.707+4*0.116,4*0.707+4*0.993);
\coordinate (k) at (2*0.707+2*0.116-2*0.325,2*0.707+2*0.993+2*0.945);
\coordinate (l) at (2*0.707+4*0.116-2*0.325,2*0.707+4*0.993+2*0.945);
\coordinate (m) at (0,2);
\coordinate (n) at (2*0.707,2*0.707+2);
\coordinate (o) at (-2*0.707,2*0.707);
\coordinate (p) at (-2*0.707,2*0.707+2);

\coordinate (q) at  (6*0.707,6*0.707);
\coordinate (r) at (2+6*0.707,6*0.707);
\coordinate (s) at (2+4*0.707+2*0.414,4*0.707+2*0.91);
\coordinate (t) at (2+6*0.707+2*0.414,6*0.707+2*0.91);

\coordinate (u) at (2+0*0.707+2*0.414,0*0.707+2*0.91);
\coordinate (v) at (2+2*0.707+2*0.414,2*0.707+2*0.91);

\coordinate (w) at (2+4*0.707+4*0.414,4*0.707+4*0.91);
\coordinate (x) at (2+6*0.707+4*0.414,6*0.707+4*0.91);
\coordinate (y) at (2+4*0.707+2*0.414-2*0.11,4*0.707+2*0.91+2*0.99);
\coordinate (z) at (2+4*0.707+4*0.414-2*0.11,4*0.707+4*0.91+2*0.99);

\draw (a) -- (b) coordinate[midway] (M);
\draw [black!20,ultra thick,postaction={decorate}] ($(M)!0mm!90:(b)$) -- ($(M)!2*0.25cm!270:(a)$);
\draw (b) -- (c) coordinate[midway] (M);
\draw [black!50,ultra thick,postaction={decorate}] ($(M)!0mm!90:(c)$) -- ($(M)!2*0.25cm!270:(b)$);
\draw [dashed] (d) -- (b);
\draw (d) -- (c) coordinate[midway] (M);
\draw [black!98,ultra thick,postaction={decorate}] ($(M)!0mm!90:(c)$) -- ($(M)!2*0.175cm!270:(d)$);
\draw (a) -- (d);

\draw (b) -- (e) coordinate[midway] (M);
\draw [black!20,ultra thick,postaction={decorate}] ($(M)!0mm!90:(e)$) -- ($(M)!2*0.25cm!270:(b)$);
\draw (e) -- (f) coordinate[midway] (M);
\draw [black!50,ultra thick,postaction={decorate}] ($(M)!0mm!90:(f)$) -- ($(M)!2*0.25cm!270:(e)$);
\draw (c) -- (f);
\draw [dashed] (c) -- (e);

\draw (b) -- (g);
\draw (g) -- (h) coordinate[midway] (M);
\draw [black!50,ultra thick,postaction={decorate}] ($(M)!0mm!90:(h)$) -- ($(M)!2*0.25cm!270:(g)$);
\draw (e) -- (h);
\draw [dashed] (e) -- (g);

\draw (g) -- (i) coordinate[midway] (M);
\draw [black!20,ultra thick,postaction={decorate}] ($(M)!0mm!90:(i)$) -- ($(M)!2*0.2cm!270:(g)$);
\draw (i) -- (j);
\draw (h) -- (j);
\draw [dashed] (h) -- (i);

\draw (g) -- (k);
\draw (k) -- (l);
\draw (i) -- (l);
\draw [dashed] (i) -- (k);

\draw (a) -- (m) coordinate[midway] (M);
\draw [black!20,ultra thick,postaction={decorate}] ($(M)!0mm!90:(m)$) -- ($(M)!2*0.25cm!270:(a)$);
\draw (m) -- (n);
\draw (n) -- (b);
\draw [dashed] (b) -- (m);

\draw (a) -- (o);
\draw (o) -- (p);
\draw (p) -- (m);
\draw [dashed] (m) -- (o);

\draw (e) -- (q);
\draw (q) -- (r);
\draw (r) -- (f) coordinate[midway] (M);
\draw [black!98,ultra thick,postaction={decorate}] ($(M)!0mm!90:(r)$) -- ($(M)!2*0.175cm!270:(f)$);
\draw [dashed] (f) -- (q);

\draw (f) -- (s); %%%%%%
\draw (s) -- (t) coordinate[midway] (M);
\draw [black!50,ultra thick,postaction={decorate}] ($(M)!0mm!90:(t)$) -- ($(M)!2*0.175cm!270:(s)$);
\draw (t) -- (r);
\draw [dashed] (r) -- (s);

\draw (d) -- (u);
\draw (u) -- (v);
\draw (v) -- (c);
\draw [dashed] (c) -- (u);

\draw (s) -- (w) coordinate[midway] (M);
\draw [black!20,ultra thick,postaction={decorate}] ($(M)!0mm!90:(w)$) -- ($(M)!2*0.25cm!270:(s)$);
\draw (w) -- (x);
\draw (x) -- (t);
\draw [dashed] (t) -- (w);

\draw [black!20] (s) -- (y);
\draw [black!20] (y) -- (z);
\draw [black!20] (z) -- (w);
\draw [black!20] [black!20,ultra thick, dashed] (w) -- (y);

\draw[black!65, line width=1.5pt, double distance=3.5pt, implies-implies] (7.5,1.7) -- (5,1.7) ;

\filldraw[black] (a) circle (2pt);
\filldraw[black] (b) circle (2pt);
\filldraw[black] (c) circle (2pt);
\filldraw[black] (d) circle (2pt);
\filldraw[black] (e) circle (2pt);
\filldraw[black] (f) circle (2pt);
\filldraw[black] (g) circle (2pt);%old v0
\filldraw[black] (h) circle (2pt);
\filldraw[black] (i) circle (2pt);
\filldraw[black] (j) circle (2pt);
\filldraw[black] (k) circle (2pt);
\filldraw[black] (l) circle (2pt);
\filldraw[black] (m) circle (2pt);
\filldraw[black] (n) circle (2pt);
\filldraw[black] (o) circle (2pt);
\filldraw[black] (p) circle (2pt);
\filldraw[black] (q) circle (2pt);
\filldraw[black] (r) circle (2pt);
\filldraw[black!25] (s) circle (2pt);%v0
\filldraw[black] (t) circle (2pt);
\filldraw[black] (u) circle (2pt);
\filldraw[black] (v) circle (2pt);
\filldraw[black] (w) circle (2pt);%v1
\filldraw[black] (x) circle (2pt);
\filldraw[black] (y) circle (2pt);%v2
\filldraw[black] (z) circle (2pt);%v3

\end{scope}
\end{tikzpicture}
%   \includegraphics[scale=3]{nd1-eps-converted-to.pdf}\\
%      \rput(4.38,5.73){\huge $N$}
%      \rput(1.8,3.43){\huge $M$}
%      \rput(4.75,1.26){\Large $H^*_{\text{root}}$}
   \begin{tikzpicture}
\begin{scope}[ultra thick, gray,  decoration={markings, mark=at position 1 with {\arrow[scale=3/2]{stealth}}}]
\coordinate (a) at (0,0);
\coordinate (b) at (0.707,0.707); %cos,sin
\coordinate (c) at (1+0.707,0.707);
\coordinate (d) at (1,0);
\coordinate (e) at (2*0.707,2*0.707);
\coordinate (f) at (1+2*0.707,2*0.707);
\coordinate (g) at (0.707+0.116,0.707+0.993);
\coordinate (h) at (2*0.707+0.116,2*0.707+0.993);
\coordinate (i) at (0.707+2*0.116,0.707+2*0.993);
\coordinate (j) at (2*0.707+2*0.116,2*0.707+2*0.993);
\coordinate (k) at (0.707+0.116-0.325,0.707+0.993+0.945);
\coordinate (l) at (0.707+2*0.116-0.325,0.707+2*0.993+0.945);
\coordinate (m) at (0,1);
\coordinate (n) at (0.707,0.707+1);
\coordinate (o) at (-0.707,0.707);
\coordinate (p) at (-0.707,0.707+1);

\coordinate (q) at  (3*0.707,3*0.707);
\coordinate (r) at (1+3*0.707,3*0.707);
\coordinate (s) at (1+2*0.707+0.91,2*0.707+0.414);
\coordinate (t) at (1+3*0.707+0.91,3*0.707+0.414);

\coordinate (u) at (1+0*0.707+0.91,0*0.707+0.414);
\coordinate (v) at (1+0.707+0.91,0.707+0.414);

\coordinate (w) at (1+2*0.707+2*0.91,2*0.707+2*0.414);
\coordinate (x) at (1+3*0.707+2*0.91,3*0.707+2*0.414);
\coordinate (y) at (1+2*0.707+0.91+0.99,2*0.707+0.414-0.11);
\coordinate (z) at (1+2*0.707+2*0.91+0.99,2*0.707+2*0.414-0.11);

\draw (a) -- (b);
\draw (b) -- (c);
\draw (d) -- (b);
\draw (d) -- (c);
\draw (a) -- (d);

\draw (b) -- (e);
\draw (e) -- (f);
\draw (c) -- (f);
\draw (c) -- (e);

\draw (b) -- (g);
\draw (g) -- (h);
\draw (e) -- (h);
\draw (e) -- (g);

\draw (g) -- (i);
\draw (i) -- (j);
\draw (h) -- (j);
\draw (h) -- (i);

\draw (g) -- (k);
\draw (k) -- (l);
\draw (i) -- (l);
\draw (i) -- (k);

\draw (a) -- (m);
\draw (m) -- (n);
\draw (n) -- (b);
\draw (b) -- (m);

\draw (a) -- (o);
\draw (o) -- (p);
\draw (p) -- (m);
\draw (m) -- (o);

\draw (e) -- (q);
\draw (q) -- (r);
\draw (r) -- (f);
\draw (f) -- (q);

\draw (f) -- (s);
\draw (s) -- (t);
\draw (t) -- (r);
\draw (r) -- (s);

\draw (d) -- (u);
\draw (u) -- (v);
\draw (v) -- (c);
\draw (c) -- (u);

\draw (s) -- (w);
\draw (w) -- (x);
\draw (x) -- (t);
\draw (t) -- (w);

\draw (s) -- (y);
\draw (y) -- (z);
\draw (z) -- (w);
\draw (w) -- (y);

\filldraw[black] (a) circle (1pt);
\filldraw[black] (b) circle (1pt);
\filldraw[black] (c) circle (1pt);
\filldraw[black] (d) circle (1pt);
\filldraw[black] (e) circle (1pt);
\filldraw[black] (f) circle (1pt);
\filldraw[black] (g) circle (1pt);%old v0
\filldraw[black] (h) circle (1pt);
\filldraw[black] (i) circle (1pt);
\filldraw[black] (j) circle (1pt);
\filldraw[black] (k) circle (1pt);
\filldraw[black] (l) circle (1pt);
\filldraw[black] (m) circle (1pt);
\filldraw[black] (n) circle (1pt);
\filldraw[black] (o) circle (1pt);
\filldraw[black] (p) circle (1pt);
\filldraw[black] (q) circle (1pt);
\filldraw[black] (r) circle (1pt);
\filldraw[black] (s) circle (1pt);%v0
\filldraw[black] (t) circle (1pt);
\filldraw[black] (u) circle (1pt);
\filldraw[black] (v) circle (1pt);
\filldraw[black] (w) circle (1pt);%v1
\filldraw[black] (x) circle (1pt);
\filldraw[black] (y) circle (1pt);%v2
\filldraw[black] (z) circle (1pt);%v3

\end{scope}
\end{tikzpicture}
   \caption{On the left: a drawing by  $\psi$ of a portion of $T^*$, viewed as the image of $H^*$ through $\pi$. On the right: the planar image of the same graph. The light gray, dark gray and black arrows represent the left $(0)$, forward $(1)$, and right $(2)$ steps between nodes in $H^*$ respectively.  The light gray colored vertex is $v=010101011\in V(T^*)$. This is the base vertex of the light gray sided rhombus  $N=H^*_{111210}$ which has a proper encoding; hence both are encoded by the QR-encoding $((2,1,0), (1,0))$  (abridged to $((2,1), (1,0))$), which is the unique encoding of $N$ (and thus the unique proper encoding of $v$). Observe that $v$ has numerous alternative improper encodings, such as $((2,1), (1))$).}
\label{fig: T star}
   \end{figure}

{\bf Encoding the rhombi.} 
To simplify our proofs, we further encode $H^*$ by a so called \emph{QR-encoding} which encodes the path to
each node via "how many forward steps to take between each turn left or right" and "whether the $i$-th turn is left or right".
Formally, writing $m(H_a)=|\{i\ :\ a_i\neq 1\}|$,
the \emph{QR-encoding} of $H_a$ consists of the pair
$\left((q_i(H_a))_{i\in [m(H_a)+1]},(\rho_i(H_a))_{i\in [m(H_a)]}\right)$.
We set 
\begin{align*}
q_i(H_a)&=\left|\{m\ :\ \sum_{j=1}^m\ind(a_j\neq 1)=i-1\}\right|-1 \\
\rho_i(H_a)&=\ind\left(\exists m\ :\ a_m=2, \sum_{j=1}^m\ind(a_j\neq 1)=i\right) 
\end{align*}
Namely,  $q_i(H_a)$ is the
number of $1$-s labeled steps between the $(i-1)$-th non-$1$ label in $a$ and
the $i$-th one, and $\rho_i(H_a)$ indicates taht the $i$-th
non-$1$ labeled step is labeled $2$. When $q_{m(N)+1}=0$ we omit it, and denote the QR encoding by 
$\left((q_i(N))_{i\in m(N)},(\rho_i(N))_{i\in m(N)}\right)$.

In accordance with our informal introduction, a QR-encoding
$\left((q_i)_{i\in[m+1]},(\rho_i)_{i\in[m]}\right)$ should be interpreted as taking $q_1$ steps
forward, then turning left or right according to $\rho_1$ being $0$ or $1$
respectively, then taking another $q_2$ steps forward in the new direction
and so on and so forth. Observe that the QR-encoding of each node is unique.

{\bf Encoding the vertices of $T^*$.}  We further require en encoding for each vertex $v\in T^*$ which would relate it to the QR encoding of a rhombus whose vertex is mapped to $v$ via $\pi$. We fist make the following observation,
\begin{obs}\label{obs: once as 2-3}
For every vertex in $v\in T^*$ except $0$ and $01$, there exists a unique $N\in H^*$ such that $\pi(N,v_j)=v$, and $j\in \{3,4\}$.
\end{obs}
To further simplify our encoding, associate each $v$ with a unique node which satisfies $\pi(N,v_0)=v$. Informally, for all vertices not contained in the root edge, this is done by taking another step forward and a step to the right or to the left from the unique rhombus in which $v$ plays the role of $v_3$ or $v_4$.

Formally, let $N\in H^*$, if $\QR(N)=\left((q_i)_{i\in[m]},(\rho_i)_{i\in[m]}\right)$ (so that $q_{m+1}=0$) and $q_{m}>0$ or $m=1$,
we say that $N$ has a \emph{proper} encoding.

This allows us to use the following observation from \cite{AF}.

\begin{obs}[{\cite[Observation~1]{AF}}]\label{obs: properties of rho-q}
Let $u\in T^*$, there exists a unique $N\in H^*$ with proper encoding such that $\pi(N,v_0)=u$.
\end{obs}
%\begin{proof}
%Firstly, observe that the only proper encodings of the vertices
%$\pi(H^*_{1},v_0)$ and $\pi(H^*_{1},v_1)$ are
%$((0,0),(0))$ and $((0,0),(1))$ respectively.

%Next, from our construction of $H^*$, we see that for every vertex $u\in T^*$, except for the base edges of the root rhombus, $\pi(H^*_{1},v_0)$ and
%$\pi(H^*_{1},v_1)$, there exists a unique node $N_u\in H^*$
%satisfying that $\pi(N_u,v_i)=u$ for some $i\in\{2,3\}$. Let $\sim$ denote
%the concatenation operation between sequences. Using this notation we have
%that either $S(N_u)\sim v_2v_3\sim v_0v_2$ or $S(N_u)\sim v_2v_3\sim v_1v_3$
%encode a node whose base vertex is mapped by $\pi$ to $u$. One may verify
%from the definition of QR-encodings that $S_N$ ending with either $v_2v_3\sim
%v_0v_2$ or with $v_2v_3\sim v_1v_3$ is equivalent to $q_{m+1}=0$ and
%$q_{m}>0$.
%\end{proof}

For example, observe that $\pi(H^*_1,v_0)=\pi(H^*_{10},v_0)$, whose QR encoding is $((0,0),(0))$. Hence this is the proper encoding of $\pi(H^*_1,v_0)=0\in T^*$.
Similarly, $\pi(H^*_{101},v_1)=\pi(H^*_{1012},v_0)$, whose QR encoding is $((0,1),(0,1))$. Hence this is the proper encoding of $\pi(H^*_{101},v_1)=01001\in T^*$.

Using this observation we define $QR(u)$ for $u\in T^*$ as the unique proper encoding of the node $N$ for which $\pi(N,v_0)=u$.

{\bf Polynomial embedding.}
Denote by $\T^d$ the $d$-dimensional unit torus in $\C^d$, i.e.,
\[\T^d=\{(x_1,\dots,x_d)\in\C^d:\forall i\in\{1,\dots,d\}, |x_i|=1\}.\]
We begin by defining a degenerate polynomial embedding, as an auxiliary in defining a full
polynomial embedding.

\begin{defin}\label{def:deg poly emb}
A \emph{ $d$-degenerate polynomial embedding of a graph $G$} is a one-to-one mapping $\psi:V(G)\to
\C[x_1,\dots,x_d]$ with integer coefficients, for which it holds that for every fixed
$x\in\T^d$ the
map $v\mapsto \psi(v)(x)=\psi_x(v)$ is a linear embedding.
\end{defin}

The importance of $d$-degenerate polynomial embeddings to our purpose stems from the
following proposition from \cite{AF}.
\begin{propos}[{\cite[Proposition 2]{AF}}]\label{prop: almost every polynomial pre-drawing is degenerate}
If $\psi$ is a $d$-degenerate polynomial embedding of a graph $G$,
then for almost every $x=(x_1,\dots,x_d)\in \T^d$, $\psi_x$ is a degenerate
drawing of $G$.
\end{propos}
%\begin{proof}
%For any $v, w\in V(G)$, the polynomials $\psi(v)(x)$ and $\psi(w)(x)$ may
%coincide only on a set of measure $0$ in $\T^d$. Taking union over all the
%pairs $v_1,v_2$, we get that outside an exceptional set of measure zero in
%$\T^d$, the map $\psi_x$ is one-to-one.
%\end{proof}

Recall our notation $(\psi_x(w),\psi_x(u))\!=\!\{\alpha\psi_x(w)+(1-\alpha)\psi_x(u):\alpha\in(0,1)\}$ and define the following.
\begin{defin}\label{def:poly emb}
A $d$-polynomial embedding of a graph $G$ is a $d$-degenerate polynomial embedding which satisfies that
for every vertex $v$ and edge $(u,w)$ in $G$ and for almost every $x\in\mathbb{T}^d$ we have
\begin{equation}
\psi_x(v)\not\in(\psi_x(w),\psi_x(u)).
\end{equation}
\end{defin}

This definition guarantees the following
\begin{obs}\label{obs: almost every good polynomial1}
If $\psi$ is a $d$-polynomial embedding of a graph $G$,
then,  for almost every \linebreak $x=(x_1,\dots,x_d)\in \T^d$, we have that $\psi_x$ is a drawing of $G$.
\end{obs}

Next, consider a vertex $v$ and an edge $(u,w)$ in $T^*$ and denote $ f(x)=\frac{\psi_x(v)-\psi_x(w)}{\psi_x(u)-\psi_x(w)}$. We first observe that if $f(x)$ is not a real function we obtain that $\psi_x(v)\in(\psi_x(w),\psi_x(u))$ only on a set of measure $0$ in $\T^d$.
Indeed, to see this, observe that $\psi_x(v)\in(\psi_x(w),\psi_x(u))$ implies that $f(x)$ is real; take a parametrization of $\T^d$ by $e^{it}=(e^{it_1},\dots,e^{it_d})$ for $t\in \R^d$ and notice that $f(e^{it})$ is a trigonometric rational function. Hence $i \Im(f(e^{it}))$ is real trigonometric rational function.  Consequently, unless $\Im(f(e^{it}))\equiv 0$, it nullifies only on a finite set in $[-\pi,\pi]^d$.

Hence we make the following observation.
\begin{obs}\label{obs: reduction1}
To show that
$\psi$ is a $d$-polynomial embedding of a graph $G$, it would suffice to show that, for every $(u,w)\in E(G)$ and $v\in V(G)$, one of the following holds:
\begin{enumerate}
\item either $\psi_x(v)\not\in(\psi_x(w),\psi_x(u))$ for all $x\in  \T^d$,
\item or the rational function $f(x)=\frac{\psi_x(v)-\psi_x(w)}{\psi_x(u)-\psi_x(w)}$ is not real on $\T^d$.
\end{enumerate}
\end{obs}

{\bf Real rational functions and central palindromicity.} In this section we establish criteria for a rational function to be real. We later apply this together with Observation~\ref{obs: reduction1} to establish the fact that our construction consists of a polynomial embedding.
Let $P$ be a multivariate polynomial with leading monomial~$N$. The ray $P+\alpha N$, for $\alpha\in\R_+$, will be called the \emph{main ray} of $P$.

Denote by $\M^d(x_1,\dots,x_n)$ the set of monomials in $x_1,\dots, x_n$ of total degree at most $d$ and write $\M=\M^{\infty}$. A rational function $g(x_1,\dots,x_n)\in \mathbb{F}(x_1,\dots,x_n)$ with coefficients in a field $\mathbb{F}$ is called \emph{central palindromic} of order $d$, if it can be written as $g\equiv\sum_{M\in \M^d(x_1,\dots,x_n)}{a_M (ML + (ML)^{-1})}$ for some real coefficients $a_M$, where $L^2\in \M^1$ (here the role of $L$ is to allow central palindromic polynomial functions containing monomials of degrees which are half integer multiples).

We say that a monomial $L_0$ is the \emph{symmetric monomial} of $L_1$ relative to $M$ if $L_0L_1\equiv M^2$. When $M$ is understandable from the context, $L_0$ will simply be called the \emph{symmetric monomial of $L_1$}. Observe that symmetric monomials in $P$ relative to $M$ have the same coefficient if and only if $P/M$ is central palindromic.

\begin{propos}\label{propos: palindromicity}
Let $P(x)$ and  $M(x)$  be a polynomial and a monomial with coefficients in $\R$, respectively. Denote
$g(x)=\frac P M (x)$. Then $g$ is real over the unit torus $\T^d$ if and only if $g$ is central palindromic, i.e., the coefficients of symmetric monomials relative to $M$ are equal.
\end{propos}
\begin{proof}
The fact that a central palindromic function $g$ is real over $\T^d$ follows directly from the fact that for all $x\in \T^d$ we have $g(x)=g(x^{-1})$ which implies that on the unit torus $g(x)=g(\bar x)=\overline{g(x)}$.

To see the other implication, let $g(x)=\frac{P(x)}{M(x)}$ be a rational function taking real values over $\T^d$ and let $d$ be the maximum total degree of $P$. Denote $P(x)=\sum_{L\in \M^d(x)} a_L L$.
Since $g(x)$ is analytic, we have $g(x)=\bar{g}(x)$ so that $g(x)=g(x^{-1})$ on $\T^d$. Hence,
\[\frac{\sum_{L\in \M^d(x)} a_L L}{M} =
\frac{\sum_{L\in \M^d(x)} a_L L^{-1}}{M^{-1}}.\]
From this we obtain,
\begin{equation}\label{eq:coef1}
\sum_{L\in \M^d(x)} a_L L M^{-1}=\sum_{L\in \M^d(x)} a_L L^{-1} M.
\end{equation}
The proposition then follows by comparing the coefficients in \eqref{eq:coef1}.
\end{proof}

We further observe that central palindromicity is preserved under substitution of monomials.
\begin{obs}\label{obs:palindromei is hereditary}
For any central palindromic function $g(x)$ and any vector $M=(M_1,\dots, M_n)$ of monomials of the variables $y=(y_1,\dots, y_m)$, the function $g(M(y))$ is a central palindromic function.
\end{obs}

We also required the following counterpart of proposition~\ref{propos: palindromicity} for quotient of a polynomial and a certain class of binomials.
\begin{propos}\label{prop:simetric factors B}
Let $P(x)$ be a complex polynomial and $M(x)$ be a monomial. Denote
$g(x)=\frac{P(x)}{M(x)(1-x_1^{-1})}$. Then $g$ is real over $\T^n$ if and only if it is central palindromic, i.e., the coefficients of monomials in $P$ which are symmetric relative to $\frac{M}{\sqrt {x_1}}$ are additive inverses.
\end{propos}
\begin{proof}
As the proof of Proposition~\ref{propos: palindromicity}, we observe that
a central palindromic function $g$ is real on $\T^d$ if and only if $g(x)=g(x^{-1})$.

To see the other implication, let $g(x)=\frac{P(x)}{M(x)(1-x_1^{-1})}$ be a rational complex function taking real values over $\T^n$ and let $d$ be the maximum total degree of $P$. Denote $P(x)=\sum_{L\in \M^d(x)} a_L L$.
As before, we have $g(x)=g(x^{-1})$ on $\T^n$ so that
\[\frac{\sum_{L\in \M^d(x)} a_L L}{M(1-x_1^{-1})} =
\frac{\sum_{L\in \M^d(x)} a_L L^{-1}}{M^{-1}(1-x_1)}.\]
Hence, multiplying by $(1-x_1)$ and using the fact that $\frac{1-x}{1-x^{-1}}=-x$, we obtain
\begin{equation}\label{eq:coef2}
\sum_{L\in \M^d(x)} a_L L \frac{\sqrt{x_1}}{M}=\sum_{L\in \M^d(x)} - a_L L^{-1} \frac{M}{\sqrt{x_1}}.
\end{equation}
The proposition then follows by comparing the coefficients in \eqref{eq:coef2}.
\end{proof}

Our construction will map vertices of $H^*$ to polynomials satisfying several properties. These are summarized in the following definition.

\begin{defin}\label{defin: 3 claimbing polynom-b}
 A polynomial $P\in \R[x_0,x_1,x_2]$  with non-negative coefficients is called \textbf{$(x_0,x_1,x_2)$-ascending} if it can be represented as
\[P=\sum_{j=1}^{m} \prod_{i=1}^{j-1} x_{i\text{ mod }3}^{a_i}\left(b_{j,0}x^{a_j-1}_{j\text{ mod }3}+b_{j,1}x^{a_j}_{j\text{ mod }3}\right)\]
%\sum_{k=1}^{a_j}b_{j,k}x^k_{j\text{ mod }3}\right)
such that for every $j\le m$, the following holds:
\begin{enumerate}
\item\label{defin: 3 claimbing polynom-b-item0} $a_j\in \N$, $b_j\in\N_0$.
\item\label{defin: 3 claimbing polynom-b-item1} $b_{m,1}>0$.
%\item\label{defin: 3 claimbing polynom-b-item2} $b_{j,k}=0$ for every $k<a_j-1$.
\item\label{defin: 3 claimbing polynom-b-item3} $a_{j}=1$ if and only if $b_{j-1,1} >0$ and $b_{j,0}=0$.
%$b_{j-1,0}>0$ and $b_{j,0}=0$.
\item\label{defin: 3 claimbing polynom-b-item4} If $a_{j}>1$ for $j<m$ then $b_{j,0}>0$.
\end{enumerate}
\end{defin}
\begin{obs}\label{obs: properties of ascending}
If $P$ is \emph{$(x_0,x_1,x_2)$-ascending} then
\begin{enumerate}
\item\label{obs: properties of ascending:it1} Monomials in $P$ are ordered by their degree so that the degrees of consecutive monomials  in each variables are weakly increasing. Denote this order by $\prec$.
We say that monomials $N,N'$ are consecutive in $P$ if they have non-zero coefficients, $N\prec N'$ and there is no $N''$ in $P$ with non-zero coefficient such that $N\prec N''\prec N'$.
\item\label{obs: properties of ascending:it2} Let $N\prec N'$ be two consecutive monomials in $P$. Then $\frac{N'}{N}$ is a bivariate monomial of the form
\begin{equation}\label{eq: purpose}
\frac{N'}{N}=x_{j\text{ mod }3}x^k_{j+1\text{ mod }3},
\end{equation}
for some $k\in \N_0$.
\item\label{obs: properties of ascending:it3} Let $N\prec N'\prec N''$ be three consecutive monomials in $P$.
Then $\frac{N'}{N}$ and $\frac{N''}{N'}$ are bivariate monomials of the form
\begin{equation*}\label{eq: purpose2}
\begin{split}
\frac{N'}{N}  &=x_{j\text{ mod }3}x^{k_1}_{j+1\text{ mod }3},\\
\frac{N''}{N'}&=x_{j+1\text{ mod }3}x^{k_2}_{j+2\text{ mod }3},
\end{split}
\end{equation*}
for some $k_1,k_2\in \N_0$.
\end{enumerate}
\end{obs}
\begin{proof} Throughout the proof we omit the modulo $3$ in the indices of $x$.
The first item is straightforward from the definition.
To see the second item, let $j\in M$ and $\ell\in \{0,1\}$ be such that
\[N=\left(\prod_{i=1}^{j-1} x_{i}^{a_i}\right)x^{a_j-1+\ell}_{j}\]
We consider two cases. If $a_{j+1}>1$, then, by Item~\ref{defin: 3 claimbing polynom-b-item4} of Definition~\ref{defin: 3 claimbing polynom-b}, we have $b_{j+1,0}> 0$ so that
\[N'\prec \left(\prod_{i=1}^{j} x_{i}^{a_i}\right)x^{a_{j+1}-1}_{j}\]
and the item follows. Otherwise $a_{j+1}=1$ so that,  by Item~\ref{defin: 3 claimbing polynom-b-item3} of Definition~\ref{defin: 3 claimbing polynom-b} applied to $j+1$, we obtain $b_{j,1}>0$, so that if $\ell=0$ the proposition follows. We are left with the case $\ell=1$. In this case, if $b_{j+1,1}=1$, we obtain
\[N'=\left(\prod_{i=1}^{j} x_{i}^{a_i}\right)x^{a_{j+1}}_{j}\]
and the second item follows.  On the other hand, if $b_{j+1,1}=0$, then, by Item~\ref{defin: 3 claimbing polynom-b-item3} of Definition~\ref{defin: 3 claimbing polynom-b} applied to $j+2$, we have $a_{j+2}>1$ so that, by items~\ref{defin: 3 claimbing polynom-b-item1}
and~\ref{defin: 3 claimbing polynom-b-item4} either $b_{j+1,0}>0$  or $b_{j+1,1}>0$. In either case we obtain
\[N'=\left(\prod_{i=1}^{j} x_{i}^{a_i}\right)x_{j+1}x^k_{j+2}\]
for some $k\in \N_0$. The second item follows.
The third item follows from similar arguments, applied to three consecutive terms.
\end{proof}

\begin{propos}\label{prop:x123-polynom-b}
Let $P,Q\in \R[x_0,x_1,x_2]$  be two $(x_0,x_1,x_2)$-ascending polynomials with coefficients over $\N_0$. %of the form
%\begin{align*}
%Q(x_0,x_1,x_2)&=\sum_{j=1}^{m} \prod_{i=1}^{j-1} x_{i\text{ mod }3}^{a_i}\left(\sum_{k=1}^{a_j}b_{j,k}x^k_{j\text{ mod }3}\right),\\
%P(x_0,x_1,x_2)&=\sum_{j=1}^{n} \prod_{i=1}^{j-1} x_{i\text{ mod }3}^{c_i}\left(\sum_{k=1}^{c_j}d_{j,k}x^k_{j\text{ mod }3}\right)
%\end{align*}
Assume that $\deg(Q)\le \deg(P)$ and that
$\frac{P-Q}{M}(x_0,x_1,x_2)$ is central palindromic, and let $M$ be the leading monomial of $Q$. Then $P-Q=kM$ for some $k\in\Z$.
\end{propos}
\begin{proof} Throughout the proof we omit the modulo $3$ in the indices of $x$.
Denote $d=\deg(M)$, by $\cP$ the collection of all
monomials with positive coefficients in $P$ which are of degree greater or equal than $d$,
and by $\bar{\cP}$ the collection monomials which are symmetric to monomials in $\cP$
with respect to $M$. Observe that the coefficient of every $N\in \cP\setminus \{M\}$ in $P-Q$ is the same
as its coefficient in $P$. By Proposition~\ref{propos: palindromicity} symmetric monomials of
$P-Q$ have the identical coefficients so that all monomials in $\bar{\cP}$ have a positive coefficient in $P-Q$.
Since the only positive contribution to monomials in $P-Q$ comes from their coefficient in $P$,
we deduce that $\bar{\cP}$ are monomials of positive coefficients in $P$.

Our purpose is to show that $\cP\subset \{M\}$, as this would imply that $P-Q$ has no monomials of positive coefficients
of degree greater than $d$. This would imply, by Proposition~\ref{propos: palindromicity}, that $P-Q=cM$ for $c\in\Z$ as
required. To this end assume towards obtaining a contradiction that $\cP\setminus \{M\}\neq \emptyset$.
Recall Item~\ref{obs: properties of ascending:it1} of Observation~\ref{obs: properties of ascending} and observe that every monomial in $\cP$ is lexicographically greater than $M$, since otherwise, its symmetric monomial will not be lexicographically ordered with respect to it.
Let $N$ be the minimal monomial greater than $M$ in $\cP$ and denote by $N'$ its symmetric monomial with respect to $M$,
Then either $N'$ and $N$ are consecutive in $P$, or $N'$ is followed by $M$ which is then followed by $N$.
The former case is impossible by Item~\ref{obs: properties of ascending:it2} of Observation~\ref{obs: properties of ascending}, since it would imply that
$\frac{N}{N'}$ is of the form $x_{j}x^k_{j+1}$ in which case $N'$ and $N$ cannot be symmetric.
The latter case could also be ruled out by observing that Item~\ref{obs: properties of ascending:it3}
of Observation~\ref{obs: properties of ascending}, does not allow $N$ and $N'$ to be symmetric.
\end{proof}

\section{Polynomial embedding of degree 12 of all outerplanar graphs}\label{sec: Two dist}

In this section we prove Proposition~\ref{prop: 13 degenerate} and thus
Theorem~\ref{thm: 13 degenerate}. To do so, we introduce in Section~\ref{subs:
psi} a $12$-polynomial embedding
$\psi=\psi(x): T^* \to \C[x]$,
where
\begin{equation}\label{eq: T defin}
x= (x_i)_{i\in I}\ \  ,\ \  I=\{0,1\}^2\times \{0,1,2\}.
\end{equation} In Section~\ref{subs: psi
img} we then write an explicit formula for the image of every vertex $v$
under $\psi$. This we do using the QR-encoding introduced in the
preliminaries section. In Section~\ref{subs: psi poly embd} we prove that
$\psi$ is a degenerate polynomial embedding.
In Section~\ref{subs: psi poly embd2} we further show that $\psi$ is
a polynomial embedding.
 Finally, in Section~\ref{subs: wrap up degenerate} we
conclude the proof of Proposition~\ref{prop: 13 degenerate}.

\subsection{The definition of \texorpdfstring{$\psi$}{psi}}\label{subs: psi} In this section we
define $\psi$. The construction is a somewhat more elaborate variant of the construction used in
\cite{AF}. We start by presenting $\psi_H(x)$, a $1$-polynomial embedding of $H$ which embeds the
rhombus graph onto a rhombus of side length $1$ with angle $x$ (identifying
the complex number $x$ with its angle on the unit circle). We then define a type function $\ty$ on each node $v\in T^*$
and a map $\psi$, which maps the rhombus encoded by $\QR(v)$ to a rotate copy of $\psi_H(x_{\ty})$ in the only way which
respects $\pi$, as defined in Section~\ref{sec:4}. The image of several subsets of $T^*$ through
$\psi$ is depicted in Figure~\ref{fig: psi T-star}.

\textbf{Polynomial embedding of a single rhombus.} We set $\psi_H(x)(v_0)=0$, $\psi_H(x)(v_1)=1$, $\psi_H(x)(v_2)=x$ and
$\psi_H(x)(v_3)=x+1$. This is indeed a polynomial embedding, mapping the
rhombus graph to a rhombus of edge length $1$, whose $v_1v_0v_2$ angle is
$x$. Figure~\ref{fig: phi} illustrates the image of $H$ under $\psi_H$.

  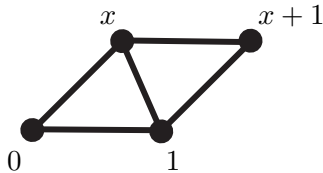
\begin{figure}[htb!]
   \centering%
   \begin{picture}(100,100)
   \put(0,0){\includegraphics[scale=3]{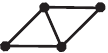}}
      \put(-5,-10){$0$}
      \put(55,-10){$1$}
      \put(30,45){$x$}
      \put(90,45){$x+1$}
   \end{picture}
   \caption{The image of $H$ under $\psi_H$. Observe how $x$ determines the $v_1v_0v_2$ angle of the rhombus.}
   \label{fig: phi}
   \end{figure}

\textbf{Defining $\ty(H)$, the type function.} Let $N$ be a vertex of $H^*$ with the $\QR$ encoding
\[QR(N)=\big((q_{i})_{i\in[m+2]}\  ,(\rho_{i})_{i\in[m+1]}\big),\] for $m\ge 0$. We set
\begin{equation}\label{eq: ty}
\ty(N)=\big(q_m(\hspace{-8pt}\mod 2),\rho_m(\hspace{-8pt} \mod 2),s_m(\hspace{-8pt}\mod 3)\big),
\end{equation}
where
\begin{equation}\label{eq: s}
s_m= \sum_{i=1}^m \ind\{q_{i-1}+\rho_i>0\},
\end{equation}
and set
$\ty(H^*_{\text{1}})=(0,0,0)$.

\textbf{Definition  of $\psi$ and $\psi_x$}.
Set $\psi(\pi(H^*_{\text{1}}))=\psi_H(x_{0,0,0})(H)$. Let $M,N\in H^*$ be
a pair of nodes such that $(M,N)$ is an arc of $H^*$, and assume that $\psi$ is
already defined on the vertices of $\pi(M)$. By the definition of $\pi$, this
implies that $\psi(\pi(N,v_0))$ and $\psi(\pi(N,v_1))$ are already defined. We
then define $\psi(\pi(N,v_2)),\psi(\pi(N,v_3))$ so that
$\psi(\pi(N,v_0)),\psi(\pi(N,v_1)),\psi(\pi(N,v_2)),\psi(\pi(N,v_3))$ form a
translated and rotated copy of $H(x_{\ty(N)})$. Because every vertex in $H^*$, except from vertices of the root edge, appears exactly once as either $v_2$ or $v_3$ in some node (Observation~\ref{obs: once as 2-3}), this is a well defined map from $T^*$ to functions from $\mathbb{T}^d$ to $\C$. For every vector $x=(x_i)_{i\in I}$ of points on the unit circle, this map reduces to a map from $T^*$ to $\C$, which we denote by $\psi_x$.

  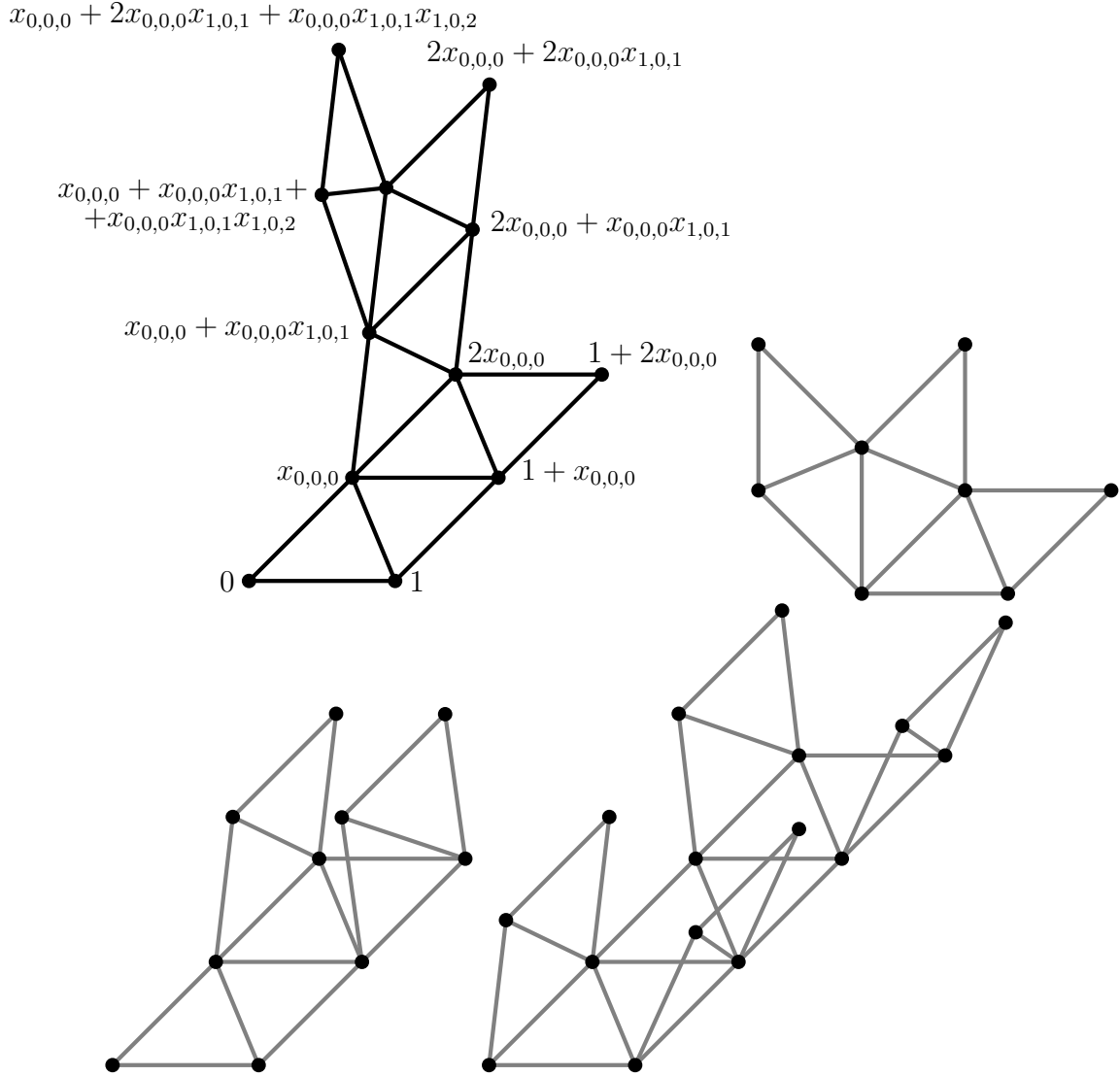
\begin{figure}[htb!]
   \centering%
%sub 4
\begin{tikzpicture}
\begin{scope}[ultra thick, black]%,  decoration={markings, mark=at position 0.65 with {\arrow[scale=3/2]{stealth}}}]
\coordinate[label={[xshift=-3mm,yshift=-3mm]\large $0$}]  (a) at (0,0);
\coordinate[label={[xshift=-6mm,yshift=-3mm]\large $x_{0,0,0}$}] (b) at (2*0.707,2*0.707); %cos,sin
\coordinate[label={[xshift=1.1cm,yshift=-3mm]\large $1+x_{0,0,0}$}] (c) at (2+2*0.707,2*0.707);
\coordinate[label={[xshift=3mm,yshift=-3mm]\large $1$}] (d) at (2,0);
\coordinate[label={[xshift=7mm,yshift=-1mm]\large $2x_{0,0,0}$}] (e) at (4*0.707,4*0.707);
\coordinate[label={[xshift=7mm,yshift=-1mm]\large $1+2x_{0,0,0}$}]  (f) at (2+4*0.707,4*0.707);
\coordinate[label={[xshift=-18mm,yshift=-3mm]\large $x_{0,0,0}+x_{0,0,0}x_{1,0,1}$}]  (g) at (2*0.707+2*0.116,2*0.707+2*0.993);
\coordinate[label={[xshift=19mm,yshift=-3mm]\large $2x_{0,0,0}+x_{0,0,0}x_{1,0,1}$}] (h) at (4*0.707+2*0.116,4*0.707+2*0.993);
\coordinate (i) at (2*0.707+4*0.116,2*0.707+4*0.993);
\coordinate[label={[xshift=9mm,yshift=0.5mm]\large $2x_{0,0,0}+2x_{0,0,0}x_{1,0,1}$}] (j) at (4*0.707+4*0.116,4*0.707+4*0.993);
\coordinate[label={[xshift=-19mm,yshift=-3mm]\large $x_{0,0,0}+x_{0,0,0}x_{1,0,1}+$}] (k) at (2*0.707+2*0.116-2*0.325,2*0.707+2*0.993+2*0.945);
\coordinate[label={[xshift=-18mm,yshift=-7mm]\large $+x_{0,0,0}x_{1,0,1}x_{1,0,2}$}] (k2) at (2*0.707+2*0.116-2*0.325,2*0.707+2*0.993+2*0.945);

\coordinate[label={[xshift=-13mm,yshift=1mm]\large $x_{0,0,0}+2x_{0,0,0}x_{1,0,1}+x_{0,0,0}x_{1,0,1}x_{1,0,2}$}] (l) at (2*0.707+4*0.116-2*0.325,2*0.707+4*0.993+2*0.945);
\draw[postaction={decorate}] (a) -- (b);
\draw[postaction={decorate}] (b) -- (c);
\draw[postaction={decorate}] (d) -- (b);
\draw[postaction={decorate}] (d) -- (c);
\draw[postaction={decorate}] (a) -- (d);

\draw[postaction={decorate}] (b) -- (e);
\draw[postaction={decorate}] (e) -- (f);
\draw[postaction={decorate}] (c) -- (f);
\draw[postaction={decorate}] (c) -- (e);

\draw[postaction={decorate}] (b) -- (g);
\draw[postaction={decorate}] (g) -- (h);
\draw[postaction={decorate}] (e) -- (h);
\draw[postaction={decorate}] (e) -- (g);

\draw[postaction={decorate}] (g) -- (i);
\draw[postaction={decorate}] (i) -- (j);
\draw[postaction={decorate}] (h) -- (j);
\draw[postaction={decorate}] (h) -- (i);

\draw[postaction={decorate}] (g) -- (k);
\draw[postaction={decorate}] (k) -- (l);
\draw[postaction={decorate}] (i) -- (l);
\draw[postaction={decorate}] (i) -- (k);

\filldraw[black] (a) circle (2pt);
\filldraw[black] (b) circle (2pt);
\filldraw[black] (c) circle (2pt);
\filldraw[black] (d) circle (2pt);
\filldraw[black] (e) circle (2pt);
\filldraw[black] (f) circle (2pt);
\filldraw[black] (g) circle (2pt);
\filldraw[black] (h) circle (2pt);
\filldraw[black] (i) circle (2pt);
\filldraw[black] (j) circle (2pt);
\filldraw[black] (k) circle (2pt);
\filldraw[black] (l) circle (2pt);

\end{scope}
\end{tikzpicture}
\begin{tikzpicture}
\begin{scope}[ultra thick, gray]%,  decoration={markings, mark=at position 0.65 with {\arrow[scale=3/2]{stealth}}}]
\coordinate (a) at (0,0);
\coordinate (b) at (2*0.707,2*0.707); %cos,sin
\coordinate (c) at (2+2*0.707,2*0.707);
\coordinate (d) at (2,0);
\coordinate (e) at (0,2);
\coordinate (f) at (2*0.707,2*0.707+2);
\coordinate (g) at (-2*0.707,2*0.707);
\coordinate (h) at (-2*0.707,2*0.707+2);
\draw[postaction={decorate}] (a) -- (b);
\draw[postaction={decorate}] (b) -- (c);
\draw[postaction={decorate}] (d) -- (b);
\draw[postaction={decorate}] (d) -- (c);
\draw[postaction={decorate}] (a) -- (d);

\draw[postaction={decorate}] (a) -- (e);
\draw[postaction={decorate}] (e) -- (f);
\draw[postaction={decorate}] (b) -- (f);
\draw[postaction={decorate}] (b) -- (e);

\draw[postaction={decorate}] (a) -- (g);
\draw[postaction={decorate}] (g) -- (h);
\draw[postaction={decorate}] (e) -- (h);
\draw[postaction={decorate}] (e) -- (g);

\filldraw[black] (a) circle (2pt);
\filldraw[black] (b) circle (2pt);
\filldraw[black] (c) circle (2pt);
\filldraw[black] (d) circle (2pt);
\filldraw[black] (e) circle (2pt);
\filldraw[black] (f) circle (2pt);
\filldraw[black] (g) circle (2pt);
\filldraw[black] (h) circle (2pt);

\end{scope}
\end{tikzpicture}
\begin{tikzpicture}
\begin{scope}[ultra thick, gray]%,  decoration={markings, mark=at position 0.65 with {\arrow[scale=3/2]{stealth}}}]
\coordinate (a) at (0,0);
\coordinate (b) at (2*0.707,2*0.707); %cos,sin
\coordinate (c) at (2+2*0.707,2*0.707);
\coordinate (d) at (2,0);
\coordinate (e) at (4*0.707,4*0.707);
\coordinate (f) at (2+4*0.707,4*0.707);
\coordinate (g) at (2*0.707+2*0.116,2*0.707+2*0.993);
\coordinate (h) at (4*0.707+2*0.116,4*0.707+2*0.993);
\coordinate (i) at (2+2*0.707-2*0.138,2*0.707+2*0.99);
\coordinate (j) at (2+4*0.707-2*0.138,4*0.707+2*0.99);
\draw[postaction={decorate}] (a) -- (b); %coordinate[midway] (M);
%\draw [black!20,ultra thick] ($(M)!0.25cm!270:(a)$) -- ($(M)!0cm!90:(b)$);
\draw[postaction={decorate}] (b) -- (c);
\draw[postaction={decorate}] (d) -- (b);
\draw[postaction={decorate}] (d) -- (c);
\draw[postaction={decorate}] (a) -- (d);

\draw[postaction={decorate}] (b) -- (e);
\draw[postaction={decorate}] (e) -- (f);
\draw[postaction={decorate}] (c) -- (f);
\draw[postaction={decorate}] (c) -- (e);

\draw[postaction={decorate}] (b) -- (g);
\draw[postaction={decorate}] (g) -- (h);
\draw[postaction={decorate}] (e) -- (h);
\draw[postaction={decorate}] (e) -- (g);

\draw[postaction={decorate}] (c) -- (i);
\draw[postaction={decorate}] (i) -- (j);
\draw[postaction={decorate}] (f) -- (j);
\draw[postaction={decorate}] (f) -- (i);

\filldraw[black] (a) circle (2pt);
\filldraw[black] (b) circle (2pt);
\filldraw[black] (c) circle (2pt);
\filldraw[black] (d) circle (2pt);
\filldraw[black] (e) circle (2pt);
\filldraw[black] (f) circle (2pt);
\filldraw[black] (g) circle (2pt);
\filldraw[black] (h) circle (2pt);
\filldraw[black] (i) circle (2pt);
\filldraw[black] (j) circle (2pt);

\end{scope}
\end{tikzpicture}
%sub3
\begin{tikzpicture}
\begin{scope}[ultra thick, gray]%,  decoration={markings, mark=at position 0.65 with {\arrow[scale=3/2]{stealth}}}]
\coordinate (a) at (0,0);
\coordinate (b) at (2*0.707,2*0.707); %cos,sin
\coordinate (c) at (2+2*0.707,2*0.707);
\coordinate (d) at (2,0);
\coordinate (e) at (4*0.707,4*0.707);
\coordinate (f) at (2+4*0.707,4*0.707);
\coordinate (i) at (0*0.707+2*0.116,0*0.707+2*0.993);
\coordinate (j) at (2*0.707+2*0.116,2*0.707+2*0.993);
\coordinate (g) at  (6*0.707,6*0.707);
\coordinate (h) at (2+6*0.707,6*0.707);
\coordinate (k) at (2+0*0.707+2*0.414,0*0.707+2*0.91);
\coordinate (l) at (2+2*0.707+2*0.414,2*0.707+2*0.91);

\coordinate (m) at (4*0.707-2*0.116,4*0.707+2*0.993);
\coordinate (n) at (6*0.707-2*0.116,6*0.707+2*0.993);
\coordinate (o) at (2+4*0.707+2*0.414,4*0.707+2*0.91);
\coordinate (p) at (2+6*0.707+2*0.414,6*0.707+2*0.91);

\draw[postaction={decorate}] (a) -- (b);
\draw[postaction={decorate}] (b) -- (c);
\draw[postaction={decorate}] (d) -- (b);
\draw[postaction={decorate}] (d) -- (c);
\draw[postaction={decorate}] (a) -- (d);

\draw[postaction={decorate}] (b) -- (e);
\draw[postaction={decorate}] (e) -- (f);
\draw[postaction={decorate}] (c) -- (f);
\draw[postaction={decorate}] (c) -- (e);

\draw[postaction={decorate}] (e) -- (g);
\draw[postaction={decorate}] (g) -- (h);
\draw[postaction={decorate}] (f) -- (h);
\draw[postaction={decorate}] (f) -- (g);

\draw[postaction={decorate}] (a) -- (i);
\draw[postaction={decorate}] (i) -- (j);
\draw[postaction={decorate}] (b) -- (j);
\draw[postaction={decorate}] (b) -- (i);

\draw[postaction={decorate}] (d) -- (k);
\draw[postaction={decorate}] (k) -- (l);
\draw[postaction={decorate}] (c) -- (l);
\draw[postaction={decorate}] (c) -- (k);

\draw[postaction={decorate}] (e) -- (m);
\draw[postaction={decorate}] (m) -- (n);
\draw[postaction={decorate}] (g) -- (n);
\draw[postaction={decorate}] (g) -- (m);

\draw[postaction={decorate}] (f) -- (o);
\draw[postaction={decorate}] (o) -- (p);
\draw[postaction={decorate}] (h) -- (p);
\draw[postaction={decorate}] (h) -- (o);

\filldraw[black] (a) circle (2pt);
\filldraw[black] (b) circle (2pt);
\filldraw[black] (c) circle (2pt);
\filldraw[black] (d) circle (2pt);
\filldraw[black] (e) circle (2pt);
\filldraw[black] (f) circle (2pt);
\filldraw[black] (g) circle (2pt);
\filldraw[black] (h) circle (2pt);
\filldraw[black] (i) circle (2pt);
\filldraw[black] (j) circle (2pt);
\filldraw[black] (k) circle (2pt);
\filldraw[black] (l) circle (2pt);
\filldraw[black] (m) circle (2pt);
\filldraw[black] (n) circle (2pt);
\filldraw[black] (o) circle (2pt);
\filldraw[black] (p) circle (2pt);
\end{scope}
\end{tikzpicture}
   \caption{The image of several subgraphs of $T^*$ under $\psi$. Explicit values are given for several vertices.}
   \label{fig: psi T-star}
   \end{figure}

As the image of every edge in $T^*$ is isometric to some edge of
$H(x_i)$ when $i\in I$, we get
\begin{obs}\label{obs: psi three edge lengths}
For every $x\in \mathbb{T}^d$, every edge of $T^*$ is mapped through $\psi_x$ to an interval of length in
\[\{1\}\cup \{|x_{i}-1|\ :\ i\in I\}.\]
\end{obs}

While this definition of $\psi((x_i)_{i\in I})$ is complete, an explicit formula for
every vertex in $T^*$ under $\psi((x_i)_{i\in I})$ is required for proving that
$\psi$ is indeed a polynomial embedding. We devote the next section to
develop this formula.

\subsection{The image of \texorpdfstring{$\psi$}{psi}}\label{subs: psi img}
In this section we state a formula for $\psi\circ \pi$ of every base vertex. An illustration of the image of several such vertices through $\psi\circ \pi$ is given in Figure~\ref{fig: psi T-star}.
Let $u\in T^*$ and let $N\in H^*$, such
that $\QR(N)=((q_k)_{k\in[m]},(\rho_k)_{k\in[m]})$ is the proper encoding of $u$ (as per Observation~\ref{obs: properties of rho-q}).
Let $x=(x_i)_{i\in I}$, and write $N_i$ for the node in $H^*$ which is
encoded by $((q_k)_{k\le i},(\rho_k)_{k\le i})$ (where $N_0=H^*_{1}$ corresponds to the null
sequence). Naturally, $N_m=N$.
From \eqref{eq: ty} we get
\begin{equation}\label{eq: nu-i}
\ty(N_i)=(q_{i-1}(\hspace{-8pt}\mod 2),\rho_{i-1}(\hspace{-8pt} \mod 2),s_{i-1}(\hspace{-8pt}\mod 3)).
\end{equation}

Observe that in the embedding of every $H^*$ node through $\psi$, the edges
$(v_0,v_2)$, $(v_1,v_3)$ are parallel, as are the edges $(v_0,v_1), (v_2,v_3)$.
Towards developing a formula for $\psi_x(u)$, denote $u_i^j:=\pi(N_i,v_j)$ and define
\begin{equation}\label{eq:P_i def}
P_i(x)=P^u_i(x):=\psi(u_i^1)-\psi(u_i^0)=\psi(u_{i-1}^2)-\psi(u_{i-1}^0).
\end{equation}
Notice that, by definition, $P_0(x)=1$.
Thus, $P^u_i$ is a unit vector in the direction of the edges
$(v_0,v_1),(v_2,v_3)$ in $\psi(\pi(N_i))$ which, for $i>0$, is the same as
the direction of $(v_0,v_2),(v_1,v_3)$ in $\psi(\pi(N_{i-1}))$.

With this in mind, it is possible to describe the change between $P_{i-1}$ and $P_{i}$, namely
\begin{equation}\label{eq: P formula}
P_{i}(x)=P_{i-1}(x)\cdot x_{\ty(N_{i-1})}.
\end{equation}

For $0\le i\le m$ write
\[Q_i(x)=Q^u_i(x):=\psi_{x}(u_i)=\psi_{x}(u_i^0),\] and observe that
$Q_0=\psi(\pi(H^*_\text{1},v_0))=0$. Next, we describe how to get $Q_i$ from
$Q_{i-1}$ using the QR encoding $((q_k),(\rho_k))$. By definition,
\[Q_i(x)-Q_{i-1}(x)=\psi(u_i^0)-\psi(u_{i-1}^0).\]
Thus, $Q_i(x)-Q_{i-1}(x)$ can be computed from the labels of the
edges along the path connecting $(N_{i-1},v_0)$ and $(N_i,v_0)$. Each edge
labeled $(v_2,v_3)$ contributes to this difference $P_{i}$, and thus in total
such edges contribute $q_i\cdot P_{i}$. An edge with label $(v_1,v_3)$
contributes $P_{i}/{x_{\ty(N_{i-1})}}=P_{i-1}$, while an edge labeled
$(v_0,v_2)$ does not change the base vertex at all.

Applying this to the encoding, we get that
\begin{equation*}
Q_i-Q_{i-1} = q_{i}\cdot P_{i}+\rho_{i}\cdot P_{i-1}.
\end{equation*}

Summing this over $1 \le i\le m$, we get:
\begin{align*}
\psi_{x}(u)=Q_{m}&=\sum_{i=1}^m \left(  q_i P_i + \rho_{i}P_{i-1} \right) \\
&= \rho_1+\sum_{i=1}^{m-1} \left( q_i+\rho_{i+1}\right) P_i+q_{m}P_{m}
\end{align*}

Equivalently, setting $\rho_{m+1}=0$ and $q_0=0$, we have
\begin{equation}\label{eq: psi formula}
\psi_{x}(u)= \sum_{i=0}^m \left( q_i+\rho_{i+1}\right) P_i(x)= \sum_{i=0}^m  c_i P_i(x),
\end{equation}
where
\begin{equation}\label{eq: c-i}
c_i=q_i+\rho_{i+1}.
\end{equation}
Observe that for every $u\in T^*$, $\psi_{x}(u)$ is a polynomial in
$\{x_i\}_{i\in I}$ (as $P_i$ are monomials). Also observe that the total
degree of $P_i$, which we denote by $\degr P_i$, obeys $\degr P_i=\degr
P_{i-1}+1$. Therefore $c_i$ is the coefficients of the only $i$-th degree monomial in the
polynomial $\psi_{x}(u)$.

Note that, in particular, using the above notation, Observation~\ref{obs:
properties of rho-q} and the fact that the encoding $((q_k),(\rho_k))$ is proper, yield that for $m>1$,
\begin{equation}\label{eq: c-m}
c_m=q_m>0.
\end{equation}

\subsection{Relating \texorpdfstring{$\psi$}{psi} to ascending polynomials}\label{subs: psi poly ascend}

In this section we show that, under the substitution $x_{i,j,k}=y_k$, the image of each vertex of $T^*$ under
$\psi$ forms a $(y_0,y_1,y_2)$-ascending polynomial.
To this end let $y_0,y_1,y_2\in \T^3$ and write $y=(y_0,y_1,y_2)$. %For every set $a=\{a_t\ :\ a_t\in \N_0\}_{t\in T}$
Denote $x(y)\in \T^{12}$ for the vector satisfying
\[x_{i,j,k}(y)=y_k,\]
and write
\begin{align*}
\psi^*(y)&=\psi(x(y)).
\end{align*}
\begin{propos}\label{prop: 0-1-2 climbing}
For every vertex $u$ the polynomial
$\psi^*_u(y)$ is $(y_0,y_1,y_2)$-ascending.
\end{propos}
\begin{proof}
Our purpose is to verify that $\psi^*_u(y)$ satisfies the conditions of Definition~\ref{defin: 3 claimbing polynom-b}.
Throughout the proof we omit the modulo $3$ in the indices of $y$. 
Denote $P^*_i(y)$ for the monomial of total degree $i$ in $\psi^*_u(y)$ as per $\eqref{eq:P_i def}$.
Recall~\eqref{eq: ty}, \eqref{eq: s} and \eqref{eq: P formula}, and use that fact that $s_m=\sum_{i=1}^m \ind(q_{i-1}+\rho_{i}>0)$ and $s_m\in \{s_{m-1},s_{m-1}+1\}$
to deduce that there existsa sequence $(a_i)_{i\le j}$ of natural numbers such that for all $\ell\in\N_0$,  writing $j=\min\{i\ :\ \sum_{i'=1}^i a_{i'}\ge \ell\}$
\[P^*_\ell(y)=\left(\prod_{i=1}^{j-1}y_{i}^{a_i}\right)y_{j}^{\ell-\sum_{i'=1}^{j-1}a_{i'}}.\]
By \eqref{eq: psi formula}, we deduce that $\psi^*_u(y_0,y_1,y_2)$ can be written
\[\psi^*_u(y_0,y_1,y_2)=\sum_{j=1}^{m} \left(\prod_{i=1}^{j-1} y_{i}^{a_i}\right)\left(b_{j,0}y^{a_j-1}_{j}+b_{j,1}y^{a_j}_{j}\right),\]
where, by \eqref{eq: c-m}, $b_{m,a_m}>0$, and $b_{j,0}=0$ if $a_j=1$. Thus Items~\ref{defin: 3 claimbing polynom-b-item0} and~\ref{defin: 3 claimbing polynom-b-item1} of the definition are satisfied.
Next, observe that, for $k\in\{0,1\}$,
\begin{equation}\label{bjk condition}
\text{$b_{j,k}=0$ if and only if $q_\ell+\rho_{\ell+1}=0$ for $\ell=\sum_{i=1}^{j}a_i-1+k$.}
\end{equation}

%\begin{equation}\label{bjk condition}
%\text{$b_{j,k}=0$ if and only if $q_\ell+\rho_{\ell+1}=0$ for $\ell=\sum_{i=1}^{j-1}a_i+k$.}
%\end{equation}
Further observe that, by \eqref{eq: s} and~\eqref{eq: P formula}, we have $\frac{P^*_{\ell+1}}{P^*_{\ell}}=\frac{P^*_{\ell}}{P^*_{\ell-1}}$ if and only if $s_\ell=s_{\ell-1}$,
so that
\begin{equation}\label{sl condition}
\text{$s_\ell=s_{\ell-1}+1 \pmod 3$ if and only if there exists $j\le m$ such that $\ell=\sum_{i=1}^{j}a_i$.}
\end{equation}

Next, let $j\le m$ and write $\ell=\sum_{i=1}^{j-1} a_i$.
To obtain Item~\ref{defin: 3 claimbing polynom-b-item3} in definition~\ref{defin: 3 claimbing polynom-b}, observe that, by putting together \eqref{bjk condition},
\eqref{sl condition} and the definition of $s_{\ell}$, we have
$a_j=1$ if and only if 
$b_{j-1,1}=q_{\ell}+\rho_{\ell+1}>0$ (which in turn implies $b_{j,0}=0$).

To obtain Item~\ref{defin: 3 claimbing polynom-b-item4}, assume that $a_j>1$ and denote $\ell=\sum_{i=1}^{j}a_i$. Observe that,
by \eqref{sl condition}, $s_\ell=s_{\ell-1}+1 \pmod 3$.
By the definition of $s_\ell$ this implies that $q_{\ell-1}+\rho_{\ell}> 0$ so that,
by \eqref{bjk condition}, $b_{j,0}>0$.
\end{proof}

\subsection{\texorpdfstring{$\psi$}{psi} is a
degenerate polynomial embedding}\label{subs: psi poly embd}

In this section we show that the image of the vertices of $T^*$ under $\psi$
are all distinct.
Relation \eqref{eq: psi formula} and Observation~\ref{obs:
psi three edge lengths} imply that if this is the case, then $\psi$ is a
$12$-degenerate polynomial embedding of $T^*$.

The proof is a significantly simplified version of the proof used in \cite{AF} (n.b.
that the result there is not directly applicable to our modified construction).
The main proposition of this section is the following.
\begin{propos}\label{prop: psi is polynomial embedding1}
Let $u,w\in T^*$ be two distinct vertices. Then $\psi_{x}(u)$ and
$\psi_{x}(w)$ are distinct polynomials in $x=(x_i)_{i\in I}$.
\end{propos}
\begin{proof}
Assume that $\psi_{x}(u)\equiv \psi_{x}(w)$ and
let the proper QR-encoding of the two vertices $u,w$ be
\begin{align*}
QR(u)&=((q^u_{i})_{i\in[m_u]}\  ,(\rho^u_{i})_{i\in[m_u]}),\\
QR(w)&=((q^w_{i})_{i\in[m_w]}\  ,(\rho^w_{i})_{i\in[m_w]}).
\end{align*}
Further assume that $m_u,m_w>1$, as otherwise $\psi_x$ is a constant function and the
proposition is straightforward.
As per \eqref{eq: psi formula}, we write
\begin{equation}\label{eq: psi u psi v}
\begin{split}
\psi_{x}(u)&=\sum_{i=0}^{m_u}c^u_iP_i^u,\\
\psi_{x}(w)&=\sum_{i=0}^{m_w}c^w_iP_i^w.
\end{split}
\end{equation}
Since $\psi_{x}(u)\equiv\psi_{x}(w)$ and $c^u_{m_u},c^w_{m_w}>0$ (by Observation~\ref{obs:
properties of rho-q}), we obtain that $m_u=m_w=m$. By substituting $x_i=y$ for all $i\in I$, and equating coefficients in~\eqref{eq: psi u psi v} we get
\begin{equation}\label{eq:coef-comp}
\forall i\in [m]:\ c^u_i=c^w_i,
\end{equation}
where $c^u_i= q^u_{i}+\rho^u_{i+1}$ and $c^w_i= q^w_{i}+\rho^w_{i+1}$.
Denote $k=\inf\{i\le m\ :\ q^u_i\neq q^w_i \vee \rho^u_{i+1}\ne\rho^w_{i+1}\}$ and $\ell=\inf\{i\ge k:c^u_i\ne 0\vee  c^w_i\ne 0\}$.
Assume for the sake of obtaining a contradiction that $k<\infty$ and observe that since $c^u_{m},c^w_{m}>0$, we have $\ell\le m<\infty$.
Observe that, by~\ref{eq: P formula} and the definition of $k$, we have $P^u_k\equiv P^w_k$ and denote this monomial by $P$.
Moreover, since
\[q^u_k+\rho^u_{k+1}=c^u_{k}=c^w_{k}=q^w_k+\rho^w_{k+1}\]
and either $q^u_k\neq q^u_k$, or $\rho^u_{k+1}\ne\rho^w_{k+1}$ we obtain
that both must hold so that $k<\ell$. By the definition of $\ell$,
for every $k<i<\ell$ we have
\[q^u_i=\rho^u_{i+1}=q^w_i=\rho^w_{i+1}=0.\]
Hence, using \eqref{eq: s}, \eqref{eq: nu-i} and \eqref{eq: P formula},
\[P^u_{\ell}=Px_{q^u_{k},\rho^u_{k},s^u_{k}}x_{0,\rho^u_{k+1},s^u_{k+1}}x_{0,0,s^u_{k+1}}^{\ell-k-2}, P^w_{\ell}=Px_{q^w_{k},\rho^w_{k},s^w_{k}}x_{0,\rho^w_{k+1},s^w_{k+1}}x_{0,0,s^w_{k+1}}^{\ell-k-2}.\]
%We obtain
%\[\frac{P^u_{\ell}}{P^w_{\ell}}=\frac{x_{q^u_{k},\rho^u_{k},s^u_{k}}x_{0,\rho^u_{k+1},s^u_{k+1}}x_{0,0,s^u_{k+1}}^{\ell-k-2}}
%{x_{q^w_{k},\rho^w_{k},s^w_{k}}x_{0,\rho^w_{k+1},s^w_{k+1}}x_{0,0,s^w_{k+1}}^{\ell-k-2}}.\]
One may verify that since $q^u_k\neq q^u_k$ and $\rho^u_{k+1}\ne\rho^w_{k+1}$, we must have
$P_{\ell}^u\ne P_{\ell}^w$, while the coefficient of at least one of the them is not $0$, in contradiction to our assumption
that $\psi_{x}(u)\equiv\psi_{x}(w)$. Hence $k=\infty$; the proposition follows.
\end{proof}
\subsection{\texorpdfstring{$\psi$}{psi} is a
polynomial embedding}\label{subs: psi poly embd2}

In this section we study the possible configurations of a vertex under $\psi$ with respect to
the line connecting the images of two adjacent vertices in $T^*$. This is studied separately for edges which form
the side of a rhombus in $H^*$ (Proposition~\ref{prop: psi is polynomial embedding2}) and those that form a diagonal (Proposition~\ref{prop: psi is polynomial embedding3}). This, together with
Relation~\eqref{eq: psi formula} and Observation~\ref{obs:
psi three edge lengths} guarantee that $\psi$ is a
$12$-polynomial embedding of $T^*$.

%\textcolor{red}{The two main proposition of this section are the following: (to delete)}
\begin{propos}\label{prop: psi is polynomial embedding2}
Let $u,w\in T^*$ be two distinct vertices, and denote by $P(x)$ the leading monomial of $\psi_{x}(w)$.
If for all $x\in \T^{12}$ we have
\[\frac{\psi_{x}(u)-\psi_{x}(w)}{P(x)}\in \R,\] then there exists $k\in\Z$ such that
\[\psi_{x}(u) \equiv \psi_{x}(w) + kP(x).\]
\end{propos}

\begin{proof}
Denote $f:=\frac{\psi_{x}(u)-\psi_{x}(w)}{P(x)}$, and assume that $f$ is a real function on $\mathbb T^{12}$, so that, by Proposition~\ref{propos: palindromicity}, $f$ is central palindromic.
Further assume that $\psi_{x}(u)\neq \psi_{x}(w)$, as otherwise the proposition is straightforward.
Let $y_0,y_1,y_2\in \T^3$ and write $y=(y_0,y_1,y_2)$.
%For every set $a=\{a_t\ :\ a_t\in \N_0\}_{t\in T}$
Denote $x(y)\in \T^{12}$ for the vector
satisfying
\[x_{i,j,k}(y)=y_k,\]
as in Section~\ref{subs: psi poly ascend}. Denote
\begin{align*}
\psi^*_{u}(y)&=\psi_{x(y)}(u),\\
\psi^*_{w}(y)&=\psi_{x(y)}(w).
\end{align*}
Observe that the leading monomial of $\psi^*_{w}(y)$ is $P(x(y))$ and denote the total degree of $P$ by $d$.
By Proposition~\ref{prop: 0-1-2 climbing}, the polynomials $\psi^*_{u}$ and $\psi^*_{w}$ are $(y_0,y_1,y_2)$-ascending.
Denote $g(y):=f(x(y))$ and observe that, by Observation~\ref{obs:palindromei is hereditary},
$g$ is central palindromic. Hence, by
Proposition~\ref{prop:x123-polynom-b}, $g(y)=k$, for some $k\in \Z$. Hence all monomials of total degree other than $d$
in $\psi^*_{u}(y)$ and $\psi^*_{w}(y)$ coincide.

Next we wish to lift the above argument to show that all monomials of total degree other than $d$
in $\psi_u:=\psi_{x}(u)$ and $\psi_w:=\psi_{x}(w)$ also coincide.
Observe that, by
\eqref{eq: psi formula}, $\psi_{u}$ and $\psi_{w}$ have at most one monomial of each total degree.
 Let $d'< d$ and denote by $P_u^{d'}$ and $P_w^{d'}$ the unique monomial of total degree $d'$ in
$\psi_{u}$ and $\psi_{w}$ respectively (which may have coefficient $0$).
Observe that, for every $d^*>d$ there is at most one monomial of total degree $d^*$ in the polynomial $\psi_{x}(u) - \psi_{x}(w)$, and thus, by Proposition~\ref{propos: palindromicity}, this polynomial contains at most one monomial of each total degree $d'<d$.
From the fact that $\psi^*_{u}(y)$ and $\psi^*_{w}(y)$ coincide we deduce that $P_u^{d'}$ and $P_w^{d'}$ are equal.

Denote $P'$ for the monomial of total degree $d$ in $\psi_u$. It follows that
\[\psi_{x}(u) - \psi_{x}(w) = k_1 P + k_2 P',\]
for $k_1,k_2\in \Z$, where $k_1<0$ for $P'\neq P$. Since  $\frac{\psi_{x}(u) - \psi_{x}(w)}{P}$ is central palindromic we deduce that either $k_2=0$, or $P=P'$. The proposition follows.
\end{proof}

\begin{propos}\label{prop: psi is polynomial embedding3}
Let $u,w\in T^*$ be two distinct vertices, where $w$ is the vertex $v_2$ of a rhombus $H$ in $H^*$, denote $y=x_{\ty(H)}$ and write $P$ for the leading monomial of $\psi_x(w)$.
If for all $x\in\T^{12}$ we have
\[\frac{\psi_x(u)-\psi_x(w)}{P(1-y^{-1})}\in\R\]
then there exists $d\in \Z $ such that $\psi_{x}(u)=\psi_{x}(w)-dP(1-y^{-1})$.
\end{propos}
The proposition is acompanid by Figure~\ref{fig: 5}.

\begin{proof}
Denote
\begin{align*}
QR(w)&=((q^w_{i})_{i\in[m_w]}\  ,(\rho^w_{i})_{i\in[m_w]}),\\
QR(u)&=((q^u_{i})_{i\in[m_u]}\  ,(\rho^u_{i})_{i\in[m_u]}),\\
\psi_x(w)&=\sum_{i=0}^{m_w}c^w_iP_i^w,\\
\psi_x(u)&=\sum_{i=0}^{m_u}c^u_iP_i^u,\\
f&=\frac{\psi_x(u)-\psi_x(w)}{P(1-y^{-1})}.
\end{align*}
%נשים לב שכיוון ש-$H$ הוא בן אמצעי הרי ש-$q^0_{m_0}>0$.
Under the assumptions of the proposition, the function $f$ is real and thus,
\[\frac{\psi_x(u)-\psi_x(w)}{P(1-y^{-1})}\equiv \frac{\psi_{\overline x}(u)- \psi_{\overline x}(w)}{\overline P(1-y)}.\]
Observe that the coefficients of the monomials of $\psi_x(u)$ and $\psi_x(w)$ are all positive. The monomials with total degree greater than $d=\deg P$ in the polynomial $\psi_x(u)-\psi_x(w)$ have positive coefficients. Hence, by Proposition~\ref{prop:simetric factors B}, their symetric monomials around $\frac{P}{\sqrt y}$ have negative coefficients. Moreover, there is at most one monomial of each total degree in the polynomial $\psi_x(u)-\psi_x(w)$.
Denote
\begin{align*}
k&=\inf\{i\in\{0,\dots,m_w-1\}\ :\ (q^u_i,\rho^u_{i})\neq (q^w_i, \rho^w_{i})\},\\
\ell&=\min\{i> k\ : \ c^u_i\neq 0\}.
\end{align*}
Observe that $k<\infty $, unless $\psi_x(u)=\psi_x(w)$. Assume towards obtaining a contradiction that $\ell< m_w$.
By the definition of $k$ and~\eqref{eq: P formula}, the monomials $P_k^w$ and $P_k^u$ are equal and we write $P_k=P_k^w=P_k^u$. Observe that since, $c^u_{k+1}=0,\dots, c^u_{\ell-1}=0$, we have
\begin{equation}\label{eq:zeroes}
q^u_{k+1}+\rho^u_{k+2}= \dots = q^u_{\ell-1}+\rho^u_{\ell}=0.
\end{equation}
Thus, by~\eqref{eq: s}, the monomials
$P_\ell^w$ and $P_\ell^u$ satisfy:
\begin{align*}
P^u_\ell&=P_k\prod_{i=k}^{\ell-1}x_{q^u_i,\rho^u_i,s_i^u}= P_k\ \cdot x_{q^u_k,\rho^u_k,s^u_k}\ \cdot x_{0,\rho^u_{k+1},s^u_{k+1}}\ \cdot(x_{0,0,s^u_{k+1}})^{\ell-k-2},\\
P^w_\ell&=P_k\prod_{i=k}^{\ell-1}x_{q^w_i,\rho^w_i,s_i^w},
\end{align*}
From the assumption that $\ell< m_w$ and Proposition~\ref{prop:simetric factors B},
it follows that
$c^u_{\ell}-c^w_{\ell}\le 0$ and hence, $P_\ell^u=P_\ell^w$. We obtain that the total degree of the varliables $x_{*,*,s^u_k}$ in the monomials $P_\ell^u$ and $P_\ell^u$ are equal.
Moreover, the monomial $\frac{P^u_\ell}{P_k}$ does not contain terms of the form $x_{*,*,s^u_k-1}$ so that $\frac{P^w_\ell}{P_k}$ does not contain such terms as well. The contradiction will now follow from a case study of $P_\ell^u$.

Recall that $s_{k+1}=s_k +\ind(q_k+\rho_{k+1}>0)$, so that if $\rho^u_{k+1}=1$ or $q_k^u>0$, then $s_{k+1}^u=s_{k}^u+1$. Hence, in this case $x_{q^w_k,\rho^w_k,s^w_k}$ must 
be the only variable of the form $x_{*,*,s^u_k}$ in $\{x_{q^w_i,\rho^w_i,s^w_i}\}_{i\in k,\dots,\ell-1}$, so that it must be equal to $x_{q^u_k,\rho^u_k,s^u_k}$, in contradiction with the definition of $k$. On the other hand, if $\rho^u_{k+1}=0$ and $q_k^u=0$ it follows that $\rho_{i+1}^w=q_i^w=0$ for every $i\in\{k+1,\dots,\ell-1\}$. Hence,

\[P^u_\ell=P_k\prod_{i=k}^{\ell-1}x_{q^u_i,\rho^u_i,s_i^u}= P_k x_{0 ,\rho^u_k, s^u_k}(x_{0,0,s^u_{k}})^{\ell-k-1},\]
so that by~\eqref{eq: s} and \eqref{eq:zeroes}, we have $x_{q^w_k,\rho^w_k,s^w_k}=x_{q^u_k,\rho^u_k,s^u_k}$, in contradiction with our definition of $k$.
We conclude that $\ell\ge m_w$. By the fact that $c^w_{m_w}\ge 0$ and $c^u_{\ell}-c^w_{\ell}\le 0$ we obtain that $\ell=m_w$. By the antisymmetry of coefficients we deduce that $c^w_{m_w-1}> 0$. By the definition of $k$ and $\ell$ we obtain that either $\psi_x(u)=\psi_x(w)$ or $k=m_w-1$.
Therefore,

\[\psi_u-\psi_w=dP-ey^{-1}P\]
where $d,e\in\Z$. From Proposition~\ref{prop:simetric factors B}, we have $d=e$. The proposition follows.
\end{proof}

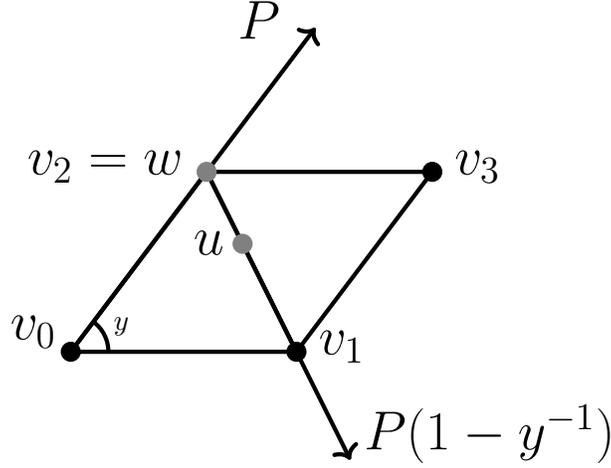
\begin{figure}
   \centering%
\begin{tikzpicture}
\draw[->,ultra thick,] (0,0) -- (3*1.8*0.602,3*1.8*0.798);
\draw[->,ultra thick,] (3*0.602,3*0.798) -- (3+0.6*3-0.6*3*0.602,-0.6*3*0.798) ;
\begin{scope}[ultra thick]%,  decoration={markings, mark=at position 0.65 with {\arrow[scale=3/2]{stealth}}}]
\coordinate[label={[xshift=-5mm,yshift=-1mm]\huge $v_0$}] (a) at (0,0);
\coordinate[label={[xshift=-13.5mm,yshift=-3mm]\huge $v_2 = w$}] (b) at (3*0.602,3*0.798); %cos,sin
\coordinate[label={[xshift=6mm,yshift=-3mm]\huge $v_3$}] (c) at (3+3*0.602,3*0.798);
\coordinate[label={[xshift=6mm,yshift=-3mm]\huge $v_1$}] (d) at (3,0);
\coordinate[label={[xshift=-4.5mm,yshift=-3mm]\huge $u$}] (e) at (9*0.602/5 +6/5, 9*0.798/5);
\coordinate[label={[xshift=-7.5mm,yshift=-3mm]\huge $P$}] (z) at (3*1.8*0.602,3*1.8*0.798);
\coordinate[label={[xshift=18.5mm,yshift=-2mm]\huge $P(1-y^{-1})$}] (y) at (3+0.6*3-0.6*3*0.602,-0.6*3*0.798);
%\coordinate (z) at (0.5*0.602,0.5*0.798);
\draw[postaction={decorate}] (a) -- (b);
\draw[postaction={decorate}] (b) -- (c);
\draw[postaction={decorate}] (d) -- (b);
\draw[postaction={decorate}] (d) -- (c);
\draw[postaction={decorate}] (a) -- (d);
\filldraw[black] (a) circle (3pt);
\filldraw[gray] (b) circle (3pt);
\filldraw[black] (c) circle (3pt);
\filldraw[black] (d) circle (3pt);
\filldraw[gray] (e) circle (3pt);

\coordinate (u) at (3*0.07*0.602,3*0.07*0.798);
\coordinate (v) at (3*0.07,0);
\pic [draw, "$y$", angle eccentricity=1.5] {angle = v--a--u};
\end{scope}
\end{tikzpicture}

\caption{
$u$ and $w$ are two distinct vertices in $T^*$, where $w$ is the vertex $v_2$ of the rhombus. $P$ is the leading monomial of $\psi_x(w)$ and $y$ is the angle between it and the base edge.%The Geometric Representation of Proposition~\ref{prop: psi is polynomial embedding3}.
%\textcolor{red}{Add $P(1-y^{-1})$}
}
\label{fig: 5}
\end{figure}

\subsection{Thirteen Distances Suffice}\label{subs: wrap up degenerate}

We are now ready to present
the proof of Proposition~\ref{prop: 13 degenerate},
and thus conclude the proof of Theorem~\ref{thm: 13 degenerate}.

\begin{proof}[Proof of Proposition~\ref{prop: 13 degenerate}]
By proposition~\ref{prop: psi is polynomial embedding1} $\psi$ is a one-to-one maping. By Definition~\ref{def:deg poly emb} and Observation~\ref{obs: psi three edge lengths} $\psi$ is a $12$-degenerate polynomial embeding of every finite subgraph $G\subseteq T^*$.

To show that $\psi$ is a polynomial embedding, we will verify that for every vertex $v$ and every edge $(u,w)\in T^*$ we have that
\begin{align}
%&\text{}\notag \\
&\text{either $\psi_x(v)\in \{\psi_x(u)+k(\psi_x(w)-\psi_x(u))\ :\ k\in \Z\}$
for all $x$ in $\mathbb{T}^d$},\notag \\
&\text{or $f(x)=\frac{\psi_x(v)-\psi_x(u)}{\psi_x(w)-\psi_x(u)}$ is not a real function on $\T^d$,} \label{eq:final conditions}
\end{align}
which would imply that $\psi$ satisfies definition~\ref{def:poly emb}, by Observation~\ref{obs: reduction1}.

Every edge $e=(u,w)$ except the base edge of $T^*$ plays either the role of a $(v_0,v_2)$ edge, or a $(v_2,v_3)$ edge, or a $(v_1,v_3)$ edge or a $(v_1,v_2)$ edge of some rhombus $N$ in $H^*$. Denote by $P$ the leading monomial of $\psi_x(v_2)$ and $y=x_{\ty H}$.
Observe that
\begin{align*}
P & =\psi_x(\pi(N,v_2))-\psi_x(\pi(N,v_0)) =\psi_x(\pi(N,v_3))-\psi_x(\pi(N,v_1)),\\
P(1-y^{-1}) & =\psi_x(\pi(N,v_2))-\psi_x(\pi(N,v_1)).
\end{align*}

Proposition~\ref{prop: psi is polynomial embedding2} guarantees that \eqref{eq:final conditions} is satisfied in case that $(u,w)$ plays the role of either $(v_0,v_2)$, $(v_2,v_3)$ or $(v_1,v_3)$ in $H$.
While Proposition~\ref{prop: psi is polynomial embedding3} guarantees that \eqref{eq:final conditions} is satisfied in case that $(u,w)$ plays the role of $(v_1,v_2)$.
Thus $\psi$ is a $12$-polynomial embedding of every finite subgraph $G\subseteq T^*$.
By Observation~\ref{obs: almost every good polynomial1}, the set of $x\in\T^{12}$ such that $\psi(x)$ is a drawing is of full measure, and by Observation~\ref{obs: psi three edge lengths} each of these drawings uses only side lengths $\{|x_i-1|\}_{i\in I}$ and $1$. To see that this is the case for almost every choice of lengths in $[0,1]$,
let $a\in(0,1)$ and observe that the embedding
$a\cdot\psi$, i.e., the composition of a multiplication by $a$ with $\psi$, is
a drawing of $G$ for almost every $x$, using the side
lengths $a, \{a|1-x_i|\}_{i\in I}$. The proposition follows.
\end{proof}

\section{Open problem}
We conclude the paper with an open problem whose solution would naturally extend our work.

In \cite{Pvpd} the original paper of Carmi, Dujmovic, Morin and Wood
the authors have shown the following theorem.
\begin{uthm}
Let $G$ be a graph with $n$ vertices, maximum degree $\Delta$, and treewidth $k$. Then the distance number of $G$ is $O(\Delta k\log n)$.
\end{uthm}

The authors use this result to provide a similar
bound for the distance number of outerplanar graphs, as these have treewidth at most $2$.
Our work demonstrates that outerplanar graphs have, in fact bounded distance number. It is therefore natural to ask
if our result could be generalized to every graph with bounded degree and treewidth.

\begin{conjecture}
Let $k,\ell\in \N$,
The distance number of all
graphs with treewidth at most $k$ and degree at most $\ell$ is bounded by $\dn(k,\ell)$,
where $\dn(k,\ell)$ is a universal constant depending only on $k$ and $\ell$.
\end{conjecture}

Naturally, it would also be of interest to find other natural families of graphs with bounded distance number.

\section{Acknowledgment}

We would like to thank an anonymous referee for useful remarks which significantly improved the presentation of our results and their context.

\end{document}